\documentclass[11pt]{article}
\usepackage[utf8]{inputenc}
\usepackage{fullpage}

\title{Combinatorial Mutations and Block Diagonal Polytopes}
\author{Oliver Clarke, Akihiro Higashitani, and Fatemeh Mohammadi }

\usepackage[square,sort,comma,numbers]{natbib}

\usepackage{amsthm}
\usepackage{amssymb}
\usepackage{amsmath}
\usepackage{pdfpages}
\usepackage{booktabs}
\usepackage{multirow}
\usepackage{graphics}
\usepackage{soul} 
\usepackage{pdflscape} 

\usepackage[bookmarksnumbered=true]{hyperref} 
\hypersetup{
     colorlinks = true,
     linkcolor = red,
     anchorcolor = blue,
     citecolor = blue,
     filecolor = blue,
     urlcolor = blue
     }
\theoremstyle{plain}

\newtheorem{theorem}{Theorem}
\newtheorem{proposition}{Proposition}
\newtheorem{corollary}{Corollary}
\newtheorem{lemma}{Lemma}

\theoremstyle{definition}

\newtheorem{definition}{Definition}
\newtheorem{example}{Example}

\theoremstyle{remark}

\newtheorem{remark}{Remark}

\newcommand{\RR}{\mathbb{R}}
\newcommand{\ZZ}{\mathbb{Z}}

\newcommand{\BL}{\mathcal{B}_\ell}
\newcommand{\MB}{\mathcal{B}}
\newcommand{\conv}{{\rm Conv}}
\def\wb{{\mathbf w}}

\newcommand{\Gr}{{\rm Gr}}
\newcommand{\MP}{{P}}
\newcommand{\init}{{\rm in}}

\begin{document}

\maketitle

\begin{abstract}
Matching fields were introduced by Sturmfels and Zelevinsky to study certain Newton polytopes and more recently have been shown to give rise to toric degenerations of various families of varieties.
Whenever a matching field gives rise to a toric degeneration, the associated polytope of the toric variety coincides with the matching field polytope.
We study combinatorial mutations, which are analogues of cluster mutations for polytopes, of matching field polytopes and show that the property of giving rise to a toric degeneration of the Grassmannians, is preserved by mutation. 
Moreover the polytopes arising through mutations 
are Newton-Okounkov bodies for the Grassmannians with respect to certain full-rank valuations.
We produce a large family of such polytopes, extending the family of so-called block diagonal matching fields.
\end{abstract}
   {
  \hypersetup{linkcolor=black}
  \tableofcontents
}
\medskip

\section{Introduction}

A toric degeneration, of a given variety $X$, is a flat family over the affine line $\mathbb A^1$ such that the fiber over $0$ is a toric variety and all other fibers are isomorphic to $X$. Toric degenerations are a valuable tool which can be used to analyze algebraic varieties \cite{An13}. They facilitate an understanding of general varieties in terms of the geometry of their associated toric varieties. Additionally, a toric variety is endowed with a polytope, or polyhedral fan, whose combinatorial data reveals geometric invariants of the toric variety. 
Moreover, toric varieties are prominent examples of irreducible varieties whose defining equations are binomials. Specialized algorithms in optimization theory and statistics have been developed to efficiently handle varieties defined by binomial equations \cite{robbiano1999computing, kahle2010decompositions, chen2017parallel}. Hence, degenerating a variety into a toric variety enables us to expand the computational techniques from toric geometry to arbitrary varieties.

Recently in \cite{kaveh2019khovanskii}, Kaveh and Manon made a direct connection between the theory of Newton-Okounkov bodies, tropical geometry and toric degenerations arising in both contexts. 
More precisely, let $X =V(I)$ be a projective variety and Trop$(X)$ its tropicalization. Points within the interior of top-dimensional cones of Trop$(X)$ are good candidates to give toric degenerations through Gr\"obner degeneration. In particular, if the corresponding initial ideal is binomial and prime, the cone is called a maximal prime cone. In this case, it is possible to construct a full-rank valuation from the prime cone and compute the Newton-Okounkov body which coincides with the polytope of the toric variety corresponding to the prime cone. More recently, \cite{escobar2019wall} Escobar and Harada studied the Newton-Okounkov bodies of adjacent maximal prime cones and described how the associated Newton-Okounkov bodies are related by so-called flip and shift operations. These are particular piecewise linear maps which are closely related to mutation. In particular, for Grassmannians Gr$(2,n)$ the wall-crossing procedure is identified with cluster mutations \cite{bossinger2020families}. In practice, it is a challenge to determine whether toric degenerations exist and, if so, how to compute them. There are a number of different constructions yielding toric degenerations for Gr$(k,n)$
including those arising from cluster algebras \cite{rietsch2017newton, bossinger2020families}, 
Gelfand-Tsetlin polytopes \cite{FvectorGC, KOGAN}, small Grassmannians such as Gr$(2,n)$, Gr$(3,n)$ \cite{speyer2004tropical, herrmann2009draw, BFFHL, KristinFatemeh}, and matching fields \cite{KristinFatemeh, clarke2019toric}.
However, the structures and relations between the polytopes arising from these toric degenerations are not well understood. In this work we use combinatorial mutations to find relations between matching field polytopes.

Combinatorial mutations were introduced by Akhtar, Coates, Galkin, and Kasprzyk in the context of mirror symmetry for Fano varieties \cite{akhtar2012minkowski} and were used to give a classification of Fano manifolds.  
Given a Fano $n$-fold $X$, a Laurent polynomial $f$ in $n$ variables is called a \textit{mirror partner} of $X$ if the classical period of $f$ coincides with the quantum period of $X$, see \cite{akhtar2012minkowski, GalkinUsnich} and the references therein. 
In \cite{GalkinUsnich}, the notion of a mutation of a Laurent polynomial was introduced as a birational transformation analogue of a cluster mutation and is shown to preserve the period of the polynomial. A \textit{combinatorial mutation} is the transformation of the Newton polytope of a Laurent polynomial undergoing mutation. This can be thought of as a kind of local transformation for lattice polytopes. 
The theory of combinatorial mutations was further developed in \cite{higashitani2020two} from a combinatorial viewpoint, and has been used to study combinatorial mutation equivalence classes of Newton-Okounkov bodies of flag varieties in \cite{fujita2020newton}. 
More specifically, it is shown that string polytopes, Nakashima-Zelevinsky polytopes and FFLV polytopes, which can all be identified as Newton-Okounkov bodies of flag varieties, are combinatorial mutation equivalent. 
Some important properties of the lattice polytopes, such as the Ehrhart polynomial, or properties of the corresponding toric varieties are preserved by combinatorial mutations. Hence, it is natural to expect that other properties are also preserved. In fact, we will see that the property of giving rise to a toric degeneration is also preserved for matching field polytopes, see Theorem~\ref{thm:toric_deg}.

For the Grassmannian Gr$(k,n)$, a \textit{matching field} is a map taking each Pl\"uker variable to a permutation, and can be interpreted a choice of initial term for the corresponding Pl\"ucker form. They were introduced by Sturmfels and Zelevinksy \cite{sturmfels1993maximal} to study the Newton polytope of a product of maximal minors of a generic matrix and have proved to be a useful tool in many contexts.  Matching fields can be viewed as a collection of perfect matchings of a complete bipartite graph. In \cite{loho2020matching}, Smith and Loho take this graph theoretic approach to study linkage matching fields and their associated Chow covectors. 
Matching fields naturally encode the image of the tropical Pl\"ucker map taking each Pl\"ucker variable to its corresponding tropical determinant. Such matching fields are called \textit{coherent}, see Definition~\ref{def:matching_field}, and are used in \cite{fink2015stiefel} to study the structure of the image of the tropical Stiefel map. The points in top-dimensional cones of the tropical Grassmannian, defined by Speyer and Sturmfels in \cite{speyer2004tropical}, parametrized by matching fields  provide a good framework for studying toric degenerations of Grassmannians. In \cite{KristinFatemeh}, the authors define the family of so-called block diagonal matching fields and show that they give rise to almost all possible Gr\"obner degenerations of Gr$(3,n)$ up to isomorphism. Moreover, block diagonal matching fields also give rise to toric degenerations for: Gr$(k,n)$ and their Schubert varieties \cite{clarke2019toric}, flag varieties Fl$_n$ \cite{OllieFatemeh2} and their Schubert varieties \cite{OllieFatemeh3}.

Each matching field $\Lambda$ admits a toric ideal $J_\Lambda$ with associated polytope $\MP_\Lambda$, see Definition~\ref{def:matching_field_polytope}. We show that understanding the polytope associated to a matching field is equivalent to finding toric degenerations of the Grassmannian as follows.

\begin{theorem}\label{thm:toric_deg}
Let $\Lambda$ be a coherent matching field for the Grassmannian $\Gr(k,n)$ with polytope $\MP_\Lambda$. If $\MP_\Lambda$ is obtained from the Gelfand-Tsetlin polytope by a sequence of combinatorial mutations, then $\Lambda$ gives rise to a toric degeneration of $\Gr(k,n)$. 
\end{theorem}

As a result we can systematically create new toric degenerations for the Grassmannian from old. In particular, we investigate the block diagonal matching fields, see Definition~\ref{def:block}. These are examples of coherent matching fields with particularly simple description. We show that all block diagonal matching field polytopes are related by a sequence of combinatorial mutations.

\begin{theorem}\label{thm:block_diag_polytope_mutation}
Any pair of block diagonal matching field polytopes can be obtained from one another by a sequence of combinatorial mutations such that all intermediate polytopes are matching field polytopes.
\end{theorem}

The matching fields associated to the intermediate polytopes can be thought of as interpolating between the block diagonal matching fields. As a result we obtain a large family of toric degenerations for the Grassmannian given by matching fields.

\medskip
\noindent{\bf Structure of the paper.}
Throughout \S\ref{sec:prelim} we introduce our main objects of study and conclude the section with a proof of Theorem~\ref{thm:toric_deg}. In \S\ref{sec:intro_mutation} we recall the definitions of a combinatorial mutation of a polytope and a mutation of its dual polytope. In \S\ref{sec:matching_fields} we recall the definition of a coherent matching field, block diagonal matching field and their ideals.
In \S\ref{sec:intermediate_matching_fields} we define intermediate matching fields whose polytopes appear in the proof of Theorem~\ref{thm:block_diag_polytope_mutation} and show that those are coherent. In \S\ref{sec:matching_field_polytopes} we define the polytopes associated to matching fields. In \S\ref{sec:intro_toric_degens} we give a proof of Theorem~\ref{thm:toric_deg} and note that the intermediate matching fields give rise to toric degenerations of the Grassmannian.

In \S\ref{sec:mutations} we give a proof of Theorem~\ref{thm:block_diag_polytope_mutation}. We begin in \S\ref{sec:mutations_gr3n} with the proof for the $\Gr(3,n)$ case which we break into three steps. Figures~\ref{fig:overview_mutate_01}, \ref{fig:overview_mutate_12} and \ref{fig:overview_mutate_ell} provide an overview for each step showing the construction of each sequence of combinatorial mutations. In \S\ref{sec:mutations_convexity} we prove some important technical results used throughout the construction of the combinatorial mutations. In \S\ref{sec:mutation_generalise_gkn} we show how the proof of the $\Gr(3,n)$ is generalised to $\Gr(k,n)$ for arbitrary $k$.

\medskip

\noindent{\bf Acknowledgement.} 
OC and FM would like to thank the organizers of the ``Workshop on Commutative Algebra and Lattice Polytopes" at RIMS in Kyoto, where this work began. OC is supported by EPSRC Doctoral Training Partnership (DTP) award EP/N509619/1. 
AH is partially supported by JSPS KAKENHI $\sharp$20K03513. 
FM was partially supported by a BOF Starting Grant of Ghent University and EPSRC Early Career Fellowship EP/R023379/1.

\section{Preliminaries}\label{sec:prelim}

\subsection{Combinatorial mutation}\label{sec:intro_mutation}
We begin by fixing two lattices $N = \ZZ^d$ and its dual $M = {\rm Hom}_\ZZ(N, \ZZ)$. We take $N_\RR = N \otimes_\ZZ \RR$ and similarly $M_\RR = M \otimes_\ZZ \RR$. We fix the standard inner product $\langle \cdot , \cdot \rangle : N_\RR \times M_\RR \rightarrow \RR$ given by evaluation $\langle v,u\rangle:=u(v)$ for $v\in N_\RR$ and $u\in M_\RR$. Let $w$ be a primitive lattice point of $M$ and $F \subset w^{\perp} \subset N_\RR$ be a lattice polytope.

\medskip

In the following, we first recall the  definition of a tropical map from \cite[\S3]{akhtar2012minkowski} which is a piecewise linear map analogous to a tropical cluster mutation.
\begin{definition} The tropical map defined by $w$ and $F$ is given by
\begin{equation*}
\varphi_{w,F}\ \colon\  M_\RR  \rightarrow M_\RR, \;\; u   \mapsto u-u_{\textrm{min}} w. 
\end{equation*}
Let $\MP \subset M_\RR$ be a lattice polytope that contains the origin and suppose that $\varphi_{w,F}(\MP)$ is convex. Then we say that the polytope $\varphi_{w,F}(\MP)$ is a \emph{combinatorial mutation} of $\MP$. \end{definition}
Since $F \subset w^\perp$ we deduce that the normalized lattice volume of $\MP$ and $\varphi_{w,F}(\MP)$ are equal.

\medskip

Given any lattice polytope $\MP \subset M_{\mathbb{R}}$, the dual polyhedron $\MP^\ast\subset N_{\mathbb{R}}$
is defined by
\[
\MP^\ast:=\{u\in  N_{\mathbb{R}}:\ \langle v,u\rangle\geq -1\text{ for all } v\in \MP \}.
\]
One can also define combinatorial mutation of the dual polytope. Note that in order to define the dual polytope, we require that the origin does not lie outside the polytope, i.e.~the origin lies in the interior of the polytope or on its boundary. Suppose $\MP \subset M_\RR$ is a polytope and $\MP^* \subset N_\RR$ is its dual. For every integer $h\in \mathbb{Z}$ we define 
\[
H_{w,h} = \{u \in M_\RR : \langle w, u \rangle = h\}\quad\text{and}\quad w_h(\MP^*) =  \conv(H_{w,h} \cap \MP^* \cap N).
\]
Note that $H_{w,h}$ is the hyperplane orthogonal to $w$ at height $h$.
Assume that for all negative integers $h \in \ZZ_{< 0}$ there exists a lattice polytope $G_h \subset N_\RR$ (with the possibility that $G_h = \emptyset$) such that
\[
H_{w,h} \cap V(\MP^*) \subseteq G_h + |h|F \subseteq w_h(\MP^*).
\] 
In such a case we can define a combinatorial mutation of the dual polytope $\MP^\ast$ as follows.

\begin{definition}
The combinatorial mutation of $\MP^*$ with respect to $w$ and $F$ is
\[
\textrm{mut}_w(\MP^*, F) = \conv\left( 
\bigcup_{h \in \ZZ_{< 0}} G_h \cup 
\bigcup_{h \in \ZZ_{\ge 0}} \Big(w_h(\MP^*) + hF \Big)
\right).
\]
\end{definition}

If the origin lies in the boundary of $\MP$ then $\MP^*$ is an unbounded polyhedron. One can define an analogous notion of combinatorial mutation by realising the polyhedron as a Minkowski sum $\MP^* = C + B$, where $C$ is a cone and $B$ is a polytope, in a canonical way. Then we first apply mutation to the polytope $B$ as described above and to the cone separately, and then take the sum to obtain $\textrm{mut}_w(\MP^*, F)$. See \cite[\S2.3]{higashitani2020two} for more details.

\subsection{Matching fields and their associated ideals}\label{sec:matching_fields}

We first define matching fields and the ideals associated to them. 
Given integers $k$ and $n$, a \emph{matching field} denoted by $\Lambda(k,n)$, or $\Lambda$ when there is no confusion, is a choice of permutation $\Lambda(I) \in S_k$ for each $I \in \mathbf{I}_{k, n} = \{ I \subset [n] : |I| = k \}$, where $[n]=\{1,2,\ldots,n\}$ for a given positive integer $n$. We think of the permutation {$\sigma=\Lambda(I)$} as inducing a new ordering on the elements of $I$, where the position of $i_s$ is  $\sigma(s)$. In addition, we think of $I$ as being identified with a monomial of the Pl\"ucker form $P_I$ and we represent these monomials as a $k \times 1$ tableau where the entry of $(\sigma(r), 1)$ is $i_{r}$. To make this tableau notation precise we define the ideal of the matching field as follows.

Let $X=(x_{i,j})$ be a $k \times n$ matrix of indeterminates. To every $k$-subset $I$ of $[n]$ with $\sigma=\Lambda(I)$ we associate the monomial 
$
\textbf{x}_{\Lambda(I)}:=x_{\sigma(1) i_{1}}x_{\sigma(2)i_2}\cdots x_{{\sigma(k)i_k}}. 
$
The {\em matching field ideal} $J_\Lambda$ is defined as the kernel of the monomial map
\begin{eqnarray}\label{eqn:monomialmap}
\varphi_{\Lambda} \colon\  & \mathbb{K}[P_I]  \rightarrow \mathbb{K}[x_{ij}]  
\quad\text{with}\quad
 P_{I}   \mapsto \text{sgn}(\Lambda(I)) \textbf{x}_{\Lambda(I)},
\end{eqnarray}
where $\text{sgn}(\Lambda(I))$ denotes the signature of the permutation $\Lambda(I)$ for each $I \in \mathbf{I}_{k, n}$. 
\begin{definition}\label{def:matching_field}
A matching field $\Lambda$  is \emph{coherent} if there exists an $k\times n$ matrix $M=(m_{ij})$ 
with 
$m_{ij}\in\mathbb{R}$ 
such that 
for every  $I \in \mathbf{I}_{k,n}$ 
the initial of the Pl\"ucker form  $P_I \in \mathbb{K}[x_{ij}]$ is $\text{in}_M (P_I) = \varphi_{\Lambda}(P_I)$, where $\text{in}_M (P_I)$ is the sum of all terms in $P_I$ of the lowest weight and the weight of a monomial $x_{1i_1}\cdots x_{ki_k}$ is $m_{1i_1}+\cdots+m_{ki_k}$.
In this case, we say that the matrix $M$ \emph{induces the matching field} $\Lambda$. We let $\wb_M$ be the weight vector on the variables $P_I$ induced by the entries $m_{ij}$ of the weight matrix $M$ on the variables $x_{ij}$. More precisely, the weight of each variable $P_I$ is defined as the minimum weight of the terms of the corresponding minor of $M$, and it is called {\em the weight induced by $M$}. 
\end{definition}

\begin{example}\label{ex:diag}
Consider the matching field $\Lambda(3,5)$ which assigns to each subset $I$ the identity permutation. Consider the following matrix:
\[
M=\begin{bmatrix}
     0  & 0  & 0  & 0  & 0  \\
     5  & 4 & 3  & 2  & 1\\
     9 & 7 & 5 & 3  & 1 \\
\end{bmatrix}.
\]
The weights induced by $M$ on the variables $P_{123}, P_{124},\ldots,P_{345}$ are $9, 7, 5, 6, 4, 3, 6, 4,3, 3$, respectively. Thus, for each $I=\{i,j,k\}$ we have that $\text{in}_M (P_I) = x_{1i}x_{2j}x_{3k}$ for $1\leq i<j<k\leq 5$. Therefore, the matrix $M$ induces $\Lambda(3,5)$. 
Below are the tableaux representing $P_I$ for each $I$:
$$\begin{array}{c}1 \\2  \\ 3 \end{array} , \quad
\begin{array}{c} 1 \\2  \\4 \end{array},\quad
\begin{array}{c}1  \\2 \\ 5 \end{array} , \quad
\begin{array}{c}1 \\ 3 \\ 4  \end{array} ,\quad
\begin{array}{c}1 \\ 3  \\ 5\end{array} ,\quad
\begin{array}{c}1 \\ 4 \\ 5\end{array} ,\quad
\begin{array}{c}2 \\ 3\\ 4  \end{array} ,\quad
\begin{array}{c}2 \\ 3\\ 5\end{array} ,\quad
\begin{array}{c} 2 \\ 4   \\ 5 \end{array} ,\quad
\begin{array}{c} 3 \\ 4   \\ 5 \end{array}. 
$$
\end{example}
Notice that each initial term $\text{in}_M(P_I)$ arises from the leading diagonal. Such matching fields are called {\em diagonal}.

\begin{definition}\label{def:block}
Given $k,n$ and $0\leq\ell\leq n$, we define the 
{\em block diagonal matching field}
$\BL$
as the map from $\mathbf{I}_{k,n}$ to $S_k$ such that
\[
 \BL(I)= \left\{
     \begin{array}{@{}l@{\thinspace}l}
      id  &: \text{if $\lvert I|=1$ or $\lvert I \cap [\ell]\rvert \neq 1$},\\
      (12)  &: \text{otherwise}. \\
     \end{array}
   \right.
\]
It is shown in \cite[Example 2.4]{clarke2019toric} that $\BL$ is a coherent matching field. In particular, it is induced by the following matrix:
\[
M_{\ell}=\begin{bmatrix}
    0       & 0         & \cdots    & 0         &  0      & 0      & \cdots     & 0  \\
    \ell    & \ell-1    & \cdots    & 1         &n        &n-1     &\cdots      &\ell+1\\
    2n      & 2(n-1)    & \cdots    & 10        &  8      & 6      &4           & 2  \\
    \vdots  & \vdots    & \ddots    & \vdots    & \vdots  & \vdots &  \vdots    &  \vdots   \\
     n(k-1) & (n-1)(k-1)& \cdots    & 5(k-1)    & 4(k-1)  & 3(k-1) & 2(k-1)     & k-1    \\
\end{bmatrix}.
\]
In order to simplify our notation we use $\wb_\ell$ for $\wb_{M_\ell}$. 

These matching fields were first called $2$-block diagonal in \cite{KristinFatemeh}. Note that $\ell=0$ or $n$ gives rise to the classical \emph{diagonal matching field} as in Example~\ref{ex:diag}.
\end{definition}

\begin{example}\label{exa:b_ell_1}
We consider the matching field $\MB_1$ with $n = 5$ and $k = 3$. We will continue this as a running example through subsequent section. The weight matrix is
\[
M_1 = 
\begin{bmatrix}
 0 & 0 & 0 & 0 & 0 \\
 1 & 5 & 4 & 3 & 2 \\
 10& 8 & 6 & 4 & 2 
\end{bmatrix}.
\]
The weight of the Pl\"ucker forms $P_{123}, P_{124}, P_{125} \dots, P_{345} \in \mathbb K[x_{i,j}]$ is given by $7,5,3,5,3,3,8,6,5,5$. And so, the tableaux representing the initial terms of the Pl\"ucker forms are
\[
\begin{array}{c}2 \\1  \\ 3 \end{array} , \quad
\begin{array}{c} 2 \\1  \\4 \end{array},\quad
\begin{array}{c}2 \\1 \\ 5 \end{array} , \quad
\begin{array}{c}3 \\ 1 \\ 4  \end{array} ,\quad
\begin{array}{c}3 \\ 1  \\ 5\end{array} ,\quad
\begin{array}{c}4 \\ 1 \\ 5\end{array} ,\quad
\begin{array}{c}2 \\ 3\\ 4  \end{array} ,\quad
\begin{array}{c}2 \\ 3\\ 5\end{array} ,\quad
\begin{array}{c} 2 \\ 4   \\ 5 \end{array} ,\quad
\begin{array}{c} 3 \\ 4   \\ 5 \end{array}. 
\]
We see that the tableaux above can be obtained from the diagonal tableaux, see Example~\ref{ex:diag}, by swapping the top two rows if the first row entry is $1$. Therefore the matching field $\MB_1$ is defined by $\MB_1(I) = (12)$ the transposition of $1$ and $2$ if $1 \in I$, otherwise $\MB_1(I) = id$ the identity permutation. The matching field ideal $J_{\MB_1}$ is a toric ideal (a prime binomial ideal) which is generated as follows
\begin{align*}
J_{\MB_1} = \langle 
& P_{135}P_{245} - P_{125}P_{345}, 
P_{134}P_{245} - P_{124}P_{345}, 
P_{135}P_{234} - P_{134}P_{235}, \\
& P_{125}P_{234} - P_{124}P_{235}, 
P_{125}P_{134} - P_{124}P_{135}
\rangle
\end{align*}
\end{example}

\subsection{Intermediate matching fields}\label{sec:intermediate_matching_fields}

\begin{definition}\label{def:int_mf}
For Grassmannian $\Gr(k,n)$ with $2 \le k \le n$, we define the matching field $\MB^{\lambda}_\ell$ for each $\ell \in \{0, \dots, n-k+1 \}$ and $\lambda \in \{\ell+2, \dots, n+\ell-1\}\setminus\{n\}$ as follows:  
Let $I=\{p,q,r_1,\ldots,r_{k-2}\}$ with $1 \leq p<q<r_1<\cdots<r_{k-2} \le n$. When $\lambda < n$, we set 
\begin{align*}
 \MB^\lambda_\ell(I)= \left\{
     \begin{array}{@{}l@{\thinspace}l}
      id  &: \text{if $q \le \ell$ or $p=\ell+1<\lambda<q$ or $\ell+1<p$},\\
      (12)  &: \text{otherwise}. \\
     \end{array}
   \right.
\end{align*}
When $\lambda > n$, we set  
\begin{align*}
 \MB^\lambda_\ell(I)= \left\{
     \begin{array}{@{}l@{\thinspace}l}
      id  &: \text{if $q \le \ell$ or $p \leq \lambda'<q=\ell+1$ or $\ell+1<p$},\\
      (12)  &: \text{otherwise}, \\
     \end{array}
   \right.
\end{align*}
where $\lambda'=\lambda-n$. 
\end{definition}
Intermediate matching fields generalise block diagonal matching fields. In particular,  \[
\MB^{n+\ell-1}_\ell = \MB_{\ell+1}\ \text{ for each }\ \ell \in \{0, \dots, n-k+1 \}.
\]
We proceed by showing that the intermediate matching fields are coherent.

\medskip

Given $n,k,\ell,\lambda$ as in Definition~\ref{def:int_mf}, we define $N = n+1$. For $\lambda<n$, let \begin{center}
\resizebox{\textwidth}{!}{
$M_\ell^\lambda = 
\begin{bmatrix}
0      & 0        & \cdots & 0 & 0                  & 0 & 0   & \cdots & 0                        & 0                      & \cdots & 0    \\
\ell & \ell - 1 & \cdots & 1 & n - \lambda + \ell+1 & n & n-1 & \cdots & n - \lambda + \ell + 2 & n - \lambda + \ell & \cdots & \ell+1 \\
Nn     & N(n-1)    & \cdots & N(n - \ell+2) & N(n-\ell+1) & N(n-\ell) & N(n-\ell-1) & \cdots & 
N(n-\lambda+1) & N(n-\lambda) & \cdots & N \\
N^2n   & N^2(n-1)   & \cdots & N^2(n-\ell+1) & N^2(n-\ell+1)   & N^2(n-\ell) & N^2(n-\ell-1) & \cdots & N^2(n-\lambda+1) & N^2(n-\lambda) & \cdots & N^2 \\
\vdots & \vdots &  & \vdots & \vdots   &  & \vdots & \vdots & \vdots & \vdots & & \vdots \\
N^{k-2}n   & N^{k-2}(n-1)   & \cdots & N^{k-2}(n-\ell+1) & N^{k-2}(n-\ell+1)   & N^{k-2}(n-\ell) & N^{k-2}(n-\ell-1) & \cdots & N^{k-2}(n-\lambda+1) & N^{k-2}(n-\lambda) & \cdots & N^{k-2} 
\end{bmatrix}$. 
}
\end{center}
Note that $n-\lambda+\ell+1$ (resp. $n-\lambda+\ell+2$) of the second row is in the $(\ell+1)$-th column (resp. the $\lambda$-column). 
Similarly, when $\lambda > n$, let 
\begin{center}
\resizebox{\textwidth}{!}{
$M_\ell^\lambda = 
\begin{bmatrix}
0    & 0        & \cdots & 0              & 0              & \cdots & 0 & 0            & 0 & 0   & \cdots & 0 \\
\ell+1 & \ell  & \cdots & \ell-\lambda'+2 & \ell-\lambda' & \cdots & 1 & \ell-\lambda'+1 & n & n-1 & \cdots & \ell + 2 \\
Nn   & N(n-1)   & \cdots & N(n-\lambda'+1) & N(n-\lambda')   & \cdots & N(n-\ell+1) & N(n-\ell) & N(n-\ell-1) & N(n-\ell-2) & \cdots & N \\
N^2n   & N^2(n-1)   & \cdots & N^2(n-\lambda'+1) & N^2(n-\lambda')   & \cdots & N^2(n-\ell+1) & N^2(n-\ell) & N^2(n-\ell-1) & N^2(n-\ell-2) & \cdots & N^2 \\
\vdots & \vdots &  & \vdots & \vdots   &  & \vdots & \vdots & \vdots & \vdots & & \vdots \\
N^{k-2}n   & N^{k-2}(n-1)   & \cdots & N^{k-2}(n-\lambda'+1) & N^{k-2}(n-\lambda')   & \cdots & N^{k-2}(n-\ell+1) & N^{k-2}(n-\ell) & N^{k-2}(n-\ell-1) & N^{k-2}(n-\ell-2) & \cdots & N^{k-2} 
\end{bmatrix}$,
}
\end{center}
where $\lambda'=n-\lambda$. Note that $\ell-\lambda'+2$ (resp. $\ell-\lambda'+1$) of the second row is in the $\lambda'$-th column (resp. the $(\ell+1)$-th column). 

\begin{proposition}
The matrix $M_\ell^\lambda$ induces the matching field $\MB_\ell^\lambda$. In particular $\MB_\ell^\lambda$ is a coherent matching field.
\end{proposition}

\begin{proof}
Let $M_\ell^\lambda=(m_{i,j})_{1 \leq i \le k, 1 \leq j \leq n}$ be the matrix defined above. We begin by showing that $M_\ell^\lambda$ induces a coherent matching field, i.e.~for each $I \in \mathbf{I}_{k, n}$ the minimal weight  induced by $M_\ell^\lambda$ is uniquely determined. 

We proceed by induction on $k$. The case $k=2$ is trivial since the entries in the second row are distinct. 
In the case $k \ge 3$, take $I=\{r_1,\ldots,r_k\} \in \mathbf{I}_{k,n}$ with $1 \le r_1<\cdots<r_k \le n$. 
Let $\alpha \in S_k$ be a permutation such that $\{m_{i,r_{\alpha(i)}} : 1 \le i \le k\}$ attains the minimal weight induced by $M_\ell^\lambda$, i.e.
\[
\textbf{w}(I) = 
\sum_{i = 1}^k m_{i,r_{\alpha(i)}} = 
\min\left\{
\sum_{i = 1}^k m_{i,r_{\beta(i)}} : \beta \in S_k
\right\}.
\]
We prove that $\alpha(k) = k$ as follows. Let $\beta \in S_k$ be any permutation. Since $m_{1,j}=m_{1,j'}=0$ and $|m_{i,j}-m_{i,j'}| < N^{i-2}n$ for each $2 \le i \le k-1$ and $1 \le j,j' \le n$, we have 
\[
\left|
\sum_{i=1}^{k-1}m_{i,r_{\alpha(i)}}-
\sum_{i=1}^{k-1}m_{i,r_{\beta(i)}}
\right| \le 
\sum_{i=1}^{k-1}
|m_{i,r_{\alpha(i)}}-m_{i,r_{\beta(i)}}| <
\sum_{i=2}^{k-1} N^{i-2}n <
N^{k-2}.
\]
Since $|m_{k,j}-m_{k,j'}| \ge N^{k-2}$ for any $1 \le j < j' \le n$,
we must have $m_{k,r_{\alpha(k)}}$ as small as possible. Since the entries of row $k$ are strictly decreasing, we have $\alpha(k)=k$ and is the unique value which minimises $w(I)$. Since the entries of row $k$ are distinct, by the  unique possibility. Hence, by the induction hypothesis, we conclude the first part of the proof.

Let $\Lambda_\ell^\lambda$ be the matching field induced by $M_\ell^\lambda$. Let $I = \{ p, q, r_1, \dots, r_{k-2}\} \in \mathbf{I}_{k,n}$ and $\sigma = \Lambda_\ell^\lambda(I)$  the permutation given by the matching field. By the above we have that $\sigma(i) = i$ for all $3 \le i \le k$. So we have that
$\sigma=id \text{ or }(12)$.
Therefore $\textbf{w}(I)$ is either 
$m_{1,p}+m_{2,q}+\sum_{i=3}^km_{i,r_{i-2}}$ (i.e.~$\sigma=id$) or $m_{1,q}+m_{2,p}+\sum_{i=3}^km_{i,r_{i-2}}$ (i.e.~$\sigma=(12)$). 
Since $m_{1,p}=m_{1,q}=0$, the permutation $\sigma$ is determined by whether $m_{2,q}<m_{2,p}$ or $m_{2,q}>m_{2,p}$. More precisely
\begin{align*}
\sigma=id \; \Longleftrightarrow \; m_{2,p}>m_{2,q} \;\text{ and }\; \sigma=(12) \; \Longleftrightarrow \; m_{2,p}<m_{2,q}. 
\end{align*}
Therefore, by the definition of $\MB_\ell^\lambda$ and the second row of $M_\ell^\lambda$, we conclude that $\MB_\ell^\lambda=\Lambda_\ell^\lambda$. 
\end{proof}

\color{black}
\subsection{Matching field polytopes}\label{sec:matching_field_polytopes}

Given a matching field $\Lambda$, we associate to it a polytope $\MP_{\Lambda}$. The vertices of the polytope are in one-to-one correspondence with the tableau of the matching field. In fact, reading the vertices of the polytope uniquely defines the matching field.

\begin{definition}\label{def:matching_field_polytope}
Fix $k$ and $n$. We take $\RR^{k \times n}$ to be the vector space of $k \times n$ matrices with canonical basis $\{ e_{i,j} : 1 \le i \le k, 1 \le j \le n\}$ where $e_{i,j}$ is the matrix with a $1$ in row $i$ and column $j$ and zeros everywhere else. Given a matching field $\Lambda$, for each $I=\{i_1,\ldots,i_k\} \in\mathbf{I}_{k,n}$ with $1 \le i_1< \cdots < i_k \le n$ we set $v_{I,\Lambda}:=\sum_{j=1}^ke_{j,i_{\Lambda(I)(j)}}$.
Then the \emph{matching field polytope} is
\[
\MP_{\Lambda} = \conv\left\{
v_{I,\Lambda}:\ I\in \mathbf{I}_{k,n}
\right\}.
\]
For notation we often write the tuple $(i_{\Lambda(I)(1)}, i_{\Lambda(I)(2)}, \dots, i_{\Lambda(I)(n)})$ for the vector $v_{I, \Lambda}$.
\end{definition}

\begin{example}
Let $k = 3$, $n = 5$ and $\Lambda$ be a matching field. Suppose $(3,1,5)^T$ is a tableau of $\Lambda$ then its corresponding vertex in $\MP_{\Lambda}$ is
\[
(3,1,5) = 
\begin{bmatrix}
 0 & 0 & 1 & 0 & 0 \\
 1 & 0 & 0 & 0 & 0 \\
 0 & 0 & 0 & 0 & 1
\end{bmatrix} 
\in \RR^{3 \times 5}.
\]
\end{example}

\begin{example}
Let $k = 3$, $\ell \in \{0, \dots, n-2 \}$ and $\lambda \in \{\ell+2, \dots, n+\ell-1\} \backslash \{n\}$. We define $\lambda' = n - \lambda$. In the table below, we write down the vertices of the matching field polytopes for all intermediate matching fields $\MB_\ell^\lambda$.

\begin{center}
\begin{tabular}{lllll}
\toprule
    \multicolumn{2}{l}{$\lambda < n$} & &
    \multicolumn{2}{l}{$\lambda > n$} \\
    $v \in \MP_{\MB_\ell^\lambda}$ & conditions & &
    $v \in \MP_{\MB_\ell^\lambda}$ & conditions\\
\midrule
    $(p,q,r)$ & $1 \le p < q < r \le n$ and $q \le \ell$ &&
    $(i,j,k)$ & $1 \le p < q < r \le n$ and $q \le \ell$ \\
    $(q,p,r)$ & $1 \le p \le \ell < q < r \le n$ &&
    $(p,\ell+1,r)$ & $1 \le p \le \lambda' < \ell+1 <r \leq n$ \\
    $(q,\ell+1,r)$ & $\ell+1<q<r \le n$ and $q \le \lambda$ &&
    $(\ell+1,p,r)$ & $p<\ell+1<r \leq n$ \\
    $(\ell+1,q,r)$ & $\lambda < q < r \le n$ &&
    $(q,p,r)$ & $1 \leq p \leq \ell+1 < q < r \le n$\\
    $(p,q,r)$ & $\ell+1 < p < q < r \le n$ &&
    $(p,q,r)$ & $\ell+1 < p < q < r \le n$\\
\bottomrule
\end{tabular}
\end{center}
\end{example}

For certain matching fields, these polytopes are in fact the toric polytopes associated to toric degenerations of the Grassmannian.

\subsection{Toric degenerations of \texorpdfstring{$\Gr(k,n)$}{Gr(k,n)} }\label{sec:intro_toric_degens}
Let $X = (x_{i,j})$ be a generic $k\times n$ matrix of indeterminates. The defining ideal of the Grassmannian, embedded into $\binom{n}{k}-1$ dimensional projective space via the Pl\"ucker embedding, is the kernel of the polynomial map
\begin{eqnarray}\label{eqn:pluckermap}
\psi \colon\  & \mathbb{K}[P_I]  \rightarrow \mathbb{K}[x_{ij}]  
\quad\text{with}\quad
 P_{I}   \mapsto [I]_X,
\end{eqnarray}
where $[I]_X$ is the determinant of the submatrix of $X = (x_{i,j})$ whose columns are given by $I$. We will denote this ideal as $G_{k,n} = \ker(\psi)$. For each $\alpha=(\alpha_J)_{J}$ in $\mathbb{Z}_{\geq 0}^{\binom{n}{k}}$ we fix the notation ${\bf P}^{{\bf \alpha}}$ denoting the monomial $\prod_{J}P_J^{\alpha_J}$.

\begin{definition}\label{def:initial}
Given a weight vector $\wb$, we denote 
the initial ideal 
of $G_{k,n}$ with respect to $\wb$ by $\init_\wb(G_{k,n})$ and we define it as the ideal generated by polynomials $\init_\wb(f)$ for all $f\in G_{k,n}$, where 
\vspace{-1mm}
\[\init_\wb(f)=\sum_{\alpha_j\cdot \wb=d}{c_{{\bf \alpha}_j}\bf P}^{{\bf \alpha}_j}\quad\text{for}\quad f=\sum_{i=1}^t c_{{\bf \alpha}_i}{\bf P}^{{\bf \alpha}_i}\quad\text{and}\quad d=\min\{\alpha_i\cdot \wb:\ i=1,\ldots,t\}.\]
\end{definition}
The Gr\"obner degeneration of $G_{k,n}$ with respect to $\wb$ is called {\em toric} if the initial ideal $\init_{\wb}(G_{k,n})$ is prime and binomial.

\medskip

The diagonal matching field gives rise to a very well-known toric degeneration of the Grassmannian, often called the Gelfand-Tsetlin degeneration. We state this well-studied fact as follows.

\begin{theorem}[Theorem 4.3 and Corollary 4.7 from \cite{clarke2019toric}]
The diagonal matching field $\MB_0$ gives rise to a toric Gr\"obner degeneration of $\Gr(k,n)$. In other words, $\init_{\wb_0}(G_{k,n}) = J_{\MB_0}$ 
is a toric ideal.
\end{theorem}

\begin{example}[Continuation of Example~\ref{exa:b_ell_1}]
For Grassmannian $\Gr(3,5)$ we have that the Pl\"ucker ideal $G_{3,5}$ is generated as follows
\begin{align*}
G_{3,5} = \langle &
P_{145}P_{235} \underline{- P_{135}P_{245} + P_{125}P_{345}}, 
P_{145}P_{234} \underline{- P_{134}P_{245} + P_{124}P_{345}}, \\ &
\underline{P_{135}P_{234} - P_{134}P_{235}} + P_{123}P_{345}, 
\underline{P_{125}P_{234} - P_{124}P_{235}} + P_{123}P_{245}, \\ &
\underline{P_{125}P_{134} - P_{124}P_{135}} + P_{123}P_{145}
\rangle
\end{align*}
In fact the above generating set is a Gr\"obner basis for $G_{3,5}$ with respect to the weight vector $\textbf{w} = (7,5,3,5,3,3,8,6,5,5)$ induced by the weight matrix $M_1$. The initial terms of the generators with respect to this weight vector are underlined. In particular we see that $\init_{\textbf{w}}(G_{3,5}) = J_{\MB_1}$
\end{example}

\begin{corollary}\label{cor:GT_vol_degree}
For each $k$ and $n$, the volume of the diagonal matching field polytope $\MP_0$ is the degree of the Grassmannian $\Gr(k,n)$.
\end{corollary}
This work is motivated by the following question from \cite{KristinFatemeh, clarke2019toric}.
\medskip

\noindent\textbf{Question.}
Which matching fields for $\Gr(k,n)$ give rise to toric degenerations?

\smallskip
This question has been studied in \cite{clarke2019toric} for the block diagonal matching fields,
and it is shown that they give rise to toric degenerations of $\Gr(k,n)$. Here, we study the polytopes of these matching fields and their combinatorial mutations. As corollaries, we give more unified and conceptual proofs of these results.

\medskip

Consider the following result about the structure of initial ideals of the Grassmannian. 

\begin{proposition}\label{thm:init_subset_of_matching_ideal}
Let $M$ be a weight matrix which induces a matching field $\Lambda$. Then $$\init_{\wb_M}(G_{k,n}) \subseteq J_\Lambda.$$
\end{proposition}

This result follows immediately from \cite[Lemma 11.3]{sturmfels1996grobner}. We now give a proof of Theorem~\ref{thm:toric_deg} which is our main tool for finding toric degenerations.

\begin{proof}[Proof of Theorem~\ref{thm:toric_deg}]
In order to show that the initial ideal of the Grassmannian is exactly the matching field ideal, we need to show the reverse inclusion to Proposition~\ref{thm:init_subset_of_matching_ideal}. To do this we consider the primary decomposition of \[
\init_{\wb_M}(G_{k,n}) = J_\Lambda \cap Q,
\]
where $Q$ is the intersection of the other components. Let us consider the degree of the varieties associated to each component. By flatness of the initial degeneration, the variety $V(\init_{\wb_M}(G_{k,n}))$ has the same degree as the Grassmannian. The degree of the toric ideal $J_\Lambda$ is the volume of its associated polytope which is the matching field polytope $\MP_\Lambda$. Since combinatorial mutations preserve volume, $\MP_\Lambda$ has the same volume as the Gelfand-Tsetlin polytope $\MP_{\MB_0}$. By Corollary~\ref{cor:GT_vol_degree}, the volume of $\MP_\Lambda$ is the degree of the Grassmannian, hence $J_\Lambda$ is the only component of $\init_{\wb_M}(G_{k,n})$. 
\end{proof}

This gives us an alternative proof that the block diagonal matching fields give rise to toric degenerations of $\Gr(k,n)$. Furthermore the proof, using Theorem~\ref{thm:toric_deg}, yields a large family of new matching fields that give rise to toric degenerations.

\begin{corollary}
The matching field $\MB^\lambda_\ell$ for each $\ell$ and $\lambda$ give rise to toric degenerations of $\Gr(k,n)$. In particular the Pl\"ucker variables form a Khovanskii basis for Pl\"ucker algebra.
\end{corollary}

In \S\ref{sec:mutations} we prove that the polytope of 
$\MB^\lambda_\ell$ is obtained from the Gelfand-Tsetlin polytope by a sequence of combinatorial mutations. Hence the above corollary follows from Theorem~\ref{thm:toric_deg}. 

\section{Mutations between matching field polytopes}\label{sec:mutations}

Throughout this section we work with linear maps between $\RR^{k \times n}$ and other vector spaces. So it is useful to define the following notation for certain projections of the matching field polytopes.
\begin{definition}
Let $V$ be a vector space over $\RR$. Let $\Pi$ be a $k \times n$ matrix whose entries $\Pi_{i,j}$ are elements of $V$. We write $\Pi$ for the linear map $\RR^{k \times n} \rightarrow V$ which takes each $e_{i,j} \in \RR^{k \times n}$ to $\Pi_{i,j}$.
\end{definition}

\begin{remark}
If the non-zero entries of $\Pi$ are linearly independent, then we can think of $\Pi$ as a projection. Let $\Pi$ be a projection and $\MP$ be a matching field polytope. Suppose that $\Pi(\MP)$ is full-dimensional and has the same dimension as $\MP$. Then there exists a linear inverse map $\Pi^{-1}$ such that $\Pi^{-1} \circ \Pi$ acts by the identity on $\MP$. In order to prove that pairs of polytopes differ by a combinatorial mutation, it is convenient to work with full-dimensional polytopes. So in this section, we construct projections with the above properties and show that the image of the projections differ by a combinatorial mutation. We note that the combinatorial mutation can be pulled back along the projections to give a mutation between the original matching field polytopes.
\end{remark}

In some cases the matching field polytopes are unimodular equivalent and the combinatorial mutations acts as a relabelling of the vertices.

\begin{definition}
A linear map $\phi : V \rightarrow V$ is called a \textit{transvection} (or \textit{shear}) if $\phi$ can be written as a matrix with $1$'s along the leading diagonal and with at most one other non-zero entry.
\end{definition}

\subsection{Grassmannian \texorpdfstring{$\Gr(3,n)$}{Gr(3,n)}}\label{sec:mutations_gr3n}

In this section we prove Theorem~\ref{thm:block_diag_polytope_mutation} for Grassmannian $\Gr(3,n)$. Throughout this section we assume that all matching fields are for $\Gr(3,n)$. We begin by showing that the polytope $\MP_{\MB_1}$ can be obtained from the Gelfand-Tsetlin polytope $P_{\MB_0}$ by a sequence of combinatorial mutations.

\begin{theorem}\label{thm:gr3n_mutate_01_block}
The block diagonal matching field polytopes $\MP_{\MB_0}$ and $\MP_{\MB_1}$ can be obtained from one another by a sequence of combinatorial mutations.
\end{theorem}
\begin{proof}
We begin by giving an overview of the structure of the proof, see Figure~\ref{fig:overview_mutate_01}. We construct a sequence of combinatorial mutations taking the polytope $\MP_{\MB_0}$ to $\MP_{\MB_1}$ which passes through the intermediate polytopes $\MP_{\MB_0^\lambda} \subseteq \RR^{3 \times n}$ for each $2 \le \lambda \le n-1$. We do this by constructing projections $\Pi_0, \Pi_0^2$ and $\Pi_1$ from $\RR^{3 \times n}$ to $\RR^{3 \times (n-3)}$ and tropical maps $\varphi_{(1,\lambda)} : \RR^{3 \times (n-3)} \rightarrow \RR^{3 \times (n-3)}$. We show that the tropical maps lift to combinatorial mutations of the matching field polytopes.

\begin{figure}
    \centering
    \includegraphics[scale=0.85]{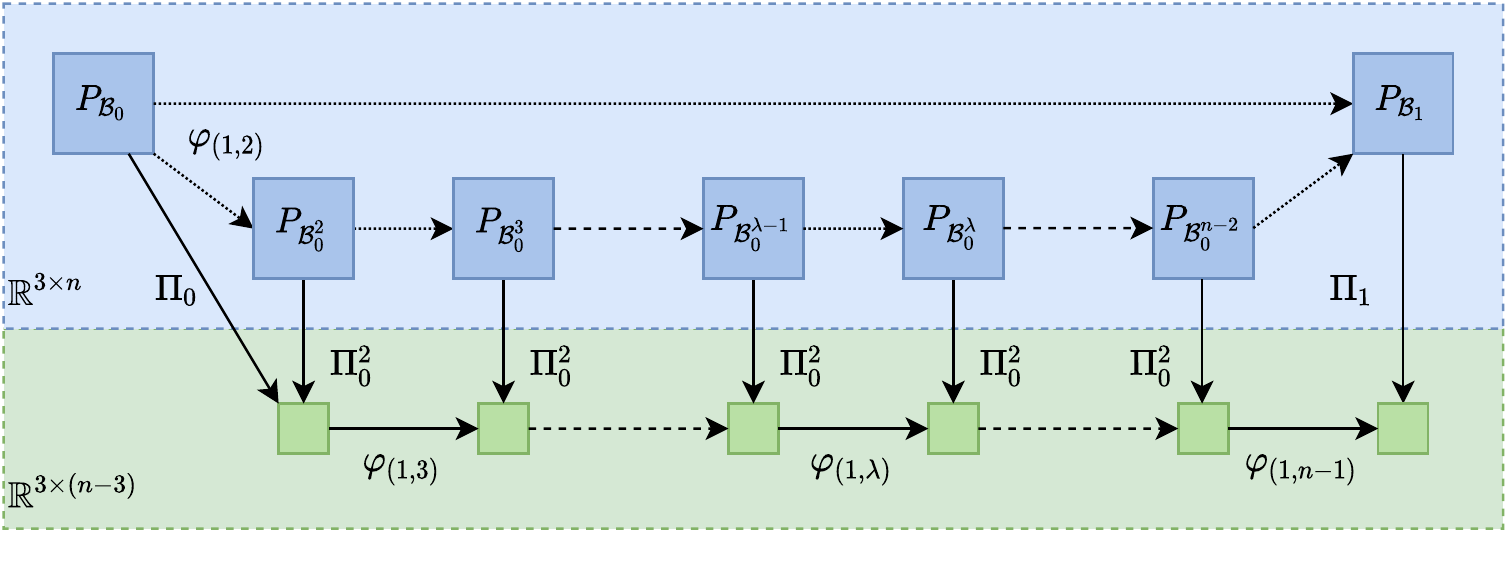} 
    \caption{Overview of the proof of Theorem~\ref{thm:gr3n_mutate_01_block}. The shaded squares represent polytopes. The squares in the top region are matching field polytopes for block diagonal and intermediate matching fields. The squares in the lower region are the images of these polytopes under the given projections.}
    \label{fig:overview_mutate_01}
\end{figure}

Consider $V = \RR^{3 \times (n-3)}$ with canonical basis $f_{i,j}$ where $1 \le i \le 3$ and $1 \le j \le n-3$. To simplify our notation we use $f_{3,n-2} = 0$ as a dummy variable which has value zero. Let $\Pi_0$ be the matrix
\[
\Pi_0 = 
\begin{bmatrix}
 0 & f_{1,1} & f_{1,2} & f_{1,3} & \dots & f_{1, n-3} & 0 & 0 \\
 0 & 0 & f_{2,1} & f_{2,2} & \dots & f_{2, n-4} & f_{2, n-3} & 0 \\
 0 & 0 & f_{3,1} & f_{3,2} & \dots & f_{2, n-4} & f_{3, n-3} & f_{3,n-2}=0 
\end{bmatrix}.
\]
Consider the diagonal matching field polytope $\MP_{\MB_0}$ whose vertices are the vectors $e_{1,i} + e_{2,j} + e_{3,k}$, which we write as $(i,j,k)$ where $1 \le i < j < k \le n$. So let us consider the vertices of $\Pi_0(\MP_{B_0})$. Note that $\Pi_0(1,2,n)=f_{3,n-2}=\underline{0}.$ We have
\begin{center}
\begin{tabular}{lll}\toprule
    $v \in V(\MP_{\MB_0})$ & $\Pi_0(v)$ & conditions \\
    \midrule
    $(1,2,k)$ & $f_{3,k-2}$ & $3 \le k \le n$\\
    $(1,j,k)$ & $f_{2,j-2} + f_{3,k-2}$ & $3 \le j < k \le n$ \\
    $(i,j,k)$ & $f_{1, i-1} + f_{2, j-2} + f_{3, k-2}$ & $2 \le i< j < k \le n$ \\
    \bottomrule
\end{tabular}
\end{center}
Note that $\MP_{\MB_0}$ is a $0/1$ polytope and the image of its vertices under $\Pi_0$ are $0/1$ vectors. It follows that the convex hull of the image of the vertices of $\MP_{\MB_0}$ coincides with $\Pi_0(\MP_{\MB_0})$. Next let $\Pi^2_0$ be the following matrix. We have highlighted the entries which are different than those of $\Pi_0$.
\[
\Pi^2_0 = 
\begin{bmatrix}
 {\bf f_{1,1}} & {\bf 0} & f_{1,2} & f_{1,3} & \dots & f_{1, n-3} & 0 & 0 \\
 0 & 0 & f_{2,1} & f_{2,2} & \dots & f_{2, n-4} & f_{2, n-3} & 0 \\
 0 & 0 & f_{3,1} & f_{3,2} & \dots & f_{2, n-4} & f_{3, n-3} & 0 
\end{bmatrix}.
\]
Similarly to the diagonal case, we write the vertices of the matching field polytope  $\MP_{\MB^2_0}$ as $(i,j,k)$. Let us, as before, consider the vertices of $\Pi^2_0(\MP_{\MB^2_0})$ which can be read off from Table~\ref{tab:vertices_under_Pi_12}.

\begin{table}
    \centering
    \begin{tabular}{lll}\toprule
        $v \in V(\MP_{\MB^{\lambda}_0})$ & $\Pi^2_0(v)$ & conditions \\
    \midrule
        $(2,1,k)$ & $f_{3,k-2}$ & $3 \le k \le n$ \\
        $(i,1,k)$ & $f_{1,i-1} + f_{3,k-2}$ & $3 \le i < k \le n$ and $i \le \lambda$ \\
        $(1,j,k)$ & $f_{1,1} + f_{2,j-2} + f_{3, k-2}$ & $\lambda < j < k \le n$ \\
        $(2,j,k)$ & $f_{2,j-2} + f_{3,k-2}$ & $3 \le j < k \le n$ \\
        $(i,j,k)$ & $f_{1, i-1} + f_{2, j-2} + f_{3, k-2}$ & $3 \le i< j < k \le n$ \\
    \bottomrule
    \end{tabular}
    \caption{Vertices of the polytopes $\Pi^2_0(\MP_{\MB^\lambda_0})$ for all $2 \le \lambda \le n-1$. Here $f_{3,n-2}=0$.}
    \label{tab:vertices_under_Pi_12}
\end{table}
We see immediately that $\Pi_0(\MP_{\MB_0}) = \Pi^2_0(\MP_{\MB^2_0})$. This means that the polytopes $\MP_{\MB^2_0}$ and $\MP_{\MB_0}$ are unimodular equivalent. We write
\[
\varphi_{(1,2)} = \left(\Pi^2_0 \right)^{-1} \circ \Pi_0 : \MP_{\MB_0} \rightarrow \MP_{\MB^2_0}
\]
for the unimodular map between them which can be thought of as a relabelling of the vertices of the polytope $\MP_{\MB_0}$.

\medskip

We now define a collection of tropical maps $\varphi_{(1,\lambda)} = \varphi_{w_{(1,\lambda)}, F_{(1,\lambda)}} : V \rightarrow V$ for $ 3 \le \lambda \le n-2$ where
\[
w_{(1,\lambda)} = \Pi^2_0 \left(
\begin{bmatrix}
-1 & 0 & \dots & 0 & 1 & 0 & \dots & 0 \\
1 & 0 & \dots & 0 & -1 & 0 & \dots & 0 \\
0 & 0 & \dots & 0 &  0 & 0 & \dots & 0
\end{bmatrix}
\right)
= -f_{1,1} +f_{1,\lambda-1} -f_{2,\lambda-2}.
\]
In the above matrix, the non-zero entries lie in columns $1$ and $\lambda$. We also have
\begin{align*}
F_{(1,\lambda)} &= \conv\left\{ \underline{0},\ 
\Pi^2_0 \left(
\begin{bmatrix}
-1 & 0 & \dots & 0 & -1 & -1 & \dots & -1 \\
 0 & 0 & \dots & 0 &  0 &  1 & \dots & 1  \\
 0 & 0 & \dots & 0 &  0 &  0 & \dots & 0  
\end{bmatrix}
\right)
\right\} \\
&= \conv \left\{ \underline{0}, \ 
-f_{1,1} +
\sum_{\mu = \lambda-1}^{n-3} (-f_{1, \mu} + f_{2, \mu})
\right\},
\end{align*}
where the columns with entries $(-1,0,0)$ are $1$ and $\lambda$. Note that $F_{(1,\lambda)} \subseteq w_{(1,\lambda)}^\perp$.

\medskip

Let us now consider the action of $\varphi_{(1,\lambda)}$ on the image of the matching field polytope $\MP_{\MB^{\lambda-1}_0}$ under $\Pi^2_0$. The vertices of this polytope are written down in Table~\ref{tab:vertices_under_phi_lambda}. From this we see that $\varphi_{(1,\lambda)}$ changes $(1,\lambda,k)$ to $(\lambda, 1, k)$ for $\lambda < k \le n$. And so, by Proposition~\ref{prop:convex}, the image of the  polytope corresponding to $\MB^{\lambda-1}_0$ is determined by the image of its vertices which is exactly the polytope corresponding to $\MB^\lambda_0$.

\begin{table}
    \centering
    \resizebox{\textwidth}{!}{
    \begin{tabular}{clll}\toprule
        $v \in V(\MP_{\MB^{\lambda - 1}_0})$ 
        & $\Pi^2_0(v)$ 
        & $\varphi_{(1,\lambda)}\left(\Pi^2_0(v)\right)$
        & conditions \\
    \midrule
        $(2,1,k)$ & $f_{3,k-2}$ & $f_{3,k-2}$ & $3 \le k \leq n$ \\
        $(i,1,k)$ & $f_{1,i-1} + f_{3,k-2}$ & $f_{1,i-1} + f_{3,k-2}$ & $3 \le i < k \leq n, \  i < \lambda$ \\
        $(1,\lambda,k)$ & $f_{1,1}+f_{2,\lambda-2}+f_{3,k-2}$ & $f_{1,\lambda-1} + f_{3,k-2}$ & $\lambda < k \leq n$ \\
    \midrule
        $(1,j,k)$ & $f_{1,1} + f_{2,j-2} + f_{3, k-2}$ & $f_{1,1} + f_{2,j-2} + f_{3, k-2}$ & $ \lambda < j < k \leq n$ \\
        $(2,j,k)$ & $f_{2,j-2} + f_{3,k-2}$ & $f_{2,j-2} + f_{3,k-2}$ & $3 \le j < k \leq n$\\ 
    \midrule
        $(i,j,k)$ & $f_{1,i-1}+f_{2,j-2}+f_{3, k-2}$ & $f_{1,i-1}+f_{2,j-2}+f_{3, k-2}$ & $3 \le i < j < k \leq n$ \\ 
    \bottomrule
    \end{tabular}
    }
    \caption{For each fixed $3 \le \lambda \le n-2$, we write down the vertices of the polytope $\varphi_{(1,\lambda)}\left(\Pi^2_0(\MP_{\MB^{\lambda-1}_0})\right)$. Note that $\Pi_0(1,2,n)=f_{3,n-2}=\underline{0}.$ }
    \label{tab:vertices_under_phi_lambda}
\end{table}

\medskip

We now define the matrix $\Pi_1$ as follows \[\Pi_1 = 
\begin{bmatrix}
0 & 0 & f_{1,2} & f_{1,3} & \dots & f_{1,n-3} & f_{1,1}   & 0 \\
0 & 0 & f_{2,1} & f_{2,2} & \dots & f_{2,n-4} & f_{2,n-3} & 0 \\
0 & 0 & f_{3,1} & f_{3,2} & \dots & f_{3,n-4} & f_{3,n-3} & 0 
\end{bmatrix}
\]
We will show that $\Pi^2_0\left( \MP_{\MB^{n-2}_0}\right)$ and $\Pi_1\left(\MP_{\MB_1}\right)$ are unimodular equivalent. The only vertices which differ are those corresponding to $(1,n-1,n)$ and $(n-1,1,n)$ respectively. We have $\Pi^2_0(1,n-1,n) = f_{1,1} + f_{2,n-3}$ and $\Pi_1(n-1,1,n) = f_{1,1}$. Let $\varphi_{(1,n-1)} : V \rightarrow V$ be the linear map defined by
\[
\varphi_{(1,n-1)}(f_{i,j}) = \left\{ 
\begin{tabular}{ll}
    $f_{1,1} - f_{2, n-3}$ &  if $(i,j) = (1,1)$,\\
    $f_{i,j}$ & otherwise. 
\end{tabular}
\right.
\]
The matrix corresponding to this map is a transvection, hence defines a unimodular transformation. Since the only vertex of the polytope $\Pi^2_0\left(\MP_{\MB^{n-2}_0}\right)$ which contains $f_{1,1}$ is the vertex corresponding to $(1,n-1,n)$, it follows that $\varphi_{(1,n-1)}$ sends $\Pi^2_0\left(\MP_{\MB^{n-2}_0}\right)$ to $\Pi_1\left(\MP_{\MB_1}\right)$ as desired.
\end{proof}

Similarly, we show that the matching field polytopes $\MP_{\MB_1}$ and $\MP_{\MB_2}$ are related by a sequence of combinatorial mutations. 

\begin{theorem}\label{thm:gr3n_mutate_12_block}
The block diagonal matching field polytope $\MP_{\MB_2}$ can be obtained from $\MP_{\MB_1}$ by a sequence of combinatorial mutations.
\end{theorem}
\begin{proof}
In Figure~\ref{fig:overview_mutate_12}, we given an overview of the structure of the proof. We construct a sequence of combinatorial mutations taking the polytope $\MP_{\MB_1}$ to $\MP_{\MB_2}$ which passes through the intermediate polytopes $\MP_{\MB_1^\lambda} \subseteq \RR^{3 \times n}$ for each $3 \le \lambda \le n-1$. We do this by constructing projections $\Pi_1, \Pi_1^3$ and $\Pi_2$ from $\RR^{3 \times n}$ to $\RR^{3 \times (n-3)}$ and tropical maps $\varphi_{(2,\lambda)} : \RR^{3 \times (n-3)} \rightarrow \RR^{3 \times (n-3)}$. We show that the tropical maps lift to combinatorial mutations of the matching field polytopes.

\begin{figure}
    \centering
    \includegraphics[scale=0.85]{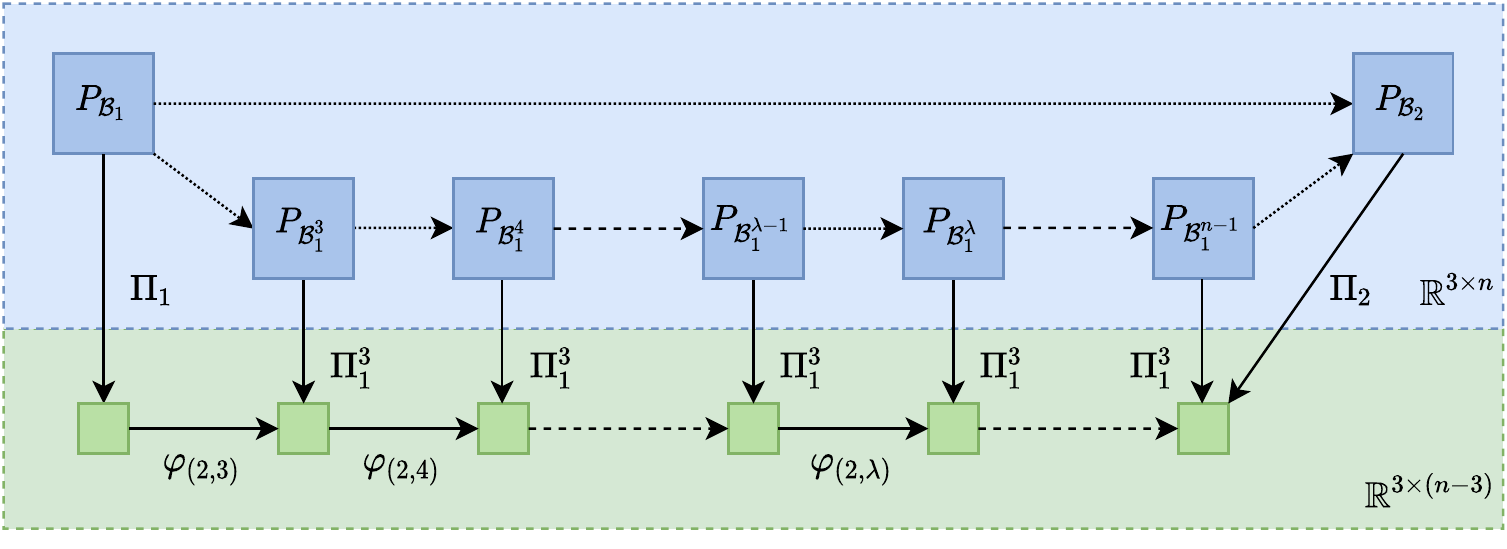} 
    \caption{Overview of the proof of Theorem~\ref{thm:gr3n_mutate_12_block}. The shaded squares represent polytopes. The squares in the top region are matching field polytopes for block diagonal and intermediate matching fields. The squares in the lower region are the images of these polytopes under the given projections.}
    \label{fig:overview_mutate_12}
\end{figure}

\medskip

Let $V = \RR^{3 \times (n-3)}$ with canonical basis $f_{i,j}$ where $1 \le i \le 3$ and $1 \le j \le n$. We define the projections
\[
\Pi_1 = 
\begin{bmatrix}
0 & 0 & f_{1,1} & f_{1,2} & \cdots & f_{1,n-3} & 0 \\
0 & \mathbf{0} & \mathbf{f_{2,1}} & f_{2,2} & \cdots & f_{2,n-3} & 0 \\
0 & 0 & f_{3,1} & f_{3,2} & \cdots & f_{3,n-3} & 0
\end{bmatrix}
\quad\text{and}\quad
\Pi^3_1 = 
\begin{bmatrix}
0 & 0 & f_{1,1} & f_{1,2} & \cdots & f_{1,n-3} & 0 \\
0 & \mathbf{f_{2,1}} & \mathbf{0} & f_{2,2} & \cdots & f_{2,n-3} & 0 \\
0 & 0 & f_{3,1} & f_{3,2} & \cdots & f_{3,n-3} & 0
\end{bmatrix}.
\]
We have highlighted the distinct entries of $\Pi_1$ and $\Pi^3_1$. Let us write down the vertices of $\MP_{\MB_1}$ under $\Pi_1$. Note that the vertex $(2,1,n)$ is mapped to $\underline{0}$ under $\Pi_1$.
\begin{center}
    \begin{tabular}{lll}\toprule
        $v \in V(\MP_{\MB_1})$ & $\Pi_1(v)$ & conditions \\
        \midrule
        $(2,1,k)$ & $f_{3,k-2}$ & $3 \le k \le n$\\
        $(i,1,k)$ & $f_{1,i-2} + f_{3, k-2}$ & $3 \le i < k \leq n$ \\
        $(2,j,k)$ & $f_{2,j-2} + f_{3,k-2}$ & $3 \le j < k \le n$ \\
        $(i,j,k)$ & $f_{1,i-1} + f_{2,j-2} + f_{3,k-2}$ & $3 \le i< j < k \le n$ \\
        \bottomrule
    \end{tabular}
\end{center}
We also write down the vertices of $\MP_{\MB^3_1}$ under $\Pi^3_1$.
\begin{center}
    \begin{tabular}{lll}\toprule
        $v \in V\left(\MP_{\MB^3_1}\right)$ & $\Pi^3_1(v)$ & conditions \\
        \midrule
        $(2,1,k)$ & $f_{3,k-2}$ & $3 \le k \le n$\\
        $(i,1,k)$ & $f_{1,i-2} + f_{3, k-2}$ & $3 \le i < k \leq n$ \\
        $(3,2,k)$ & $f_{1,1} + f_{2,1} + f_{3,k-2}$ & $4 \le k \le n$ \\
        $(2,j,k)$ & $f_{2,j-2} + f_{3,k-2}$ & $4 \le j < k \le n$ \\
        $(i,j,k)$ & $f_{1,i-1} + f_{2,j-2} + f_{3,k-2}$ & $3 \le i< j < k \le n$ \\
        \bottomrule
    \end{tabular}
\end{center}
We define the linear map $\varphi_{(2,3)} : V \rightarrow V$ as follows
\[
\varphi_{(2,3)}(f_{i,j}) = \left\{ 
\begin{tabular}{ll}
    $f_{2,1} + f_{1,1}$ &  if $(i,j) = (2,1)$\\
    $f_{i,j}$ & otherwise. 
\end{tabular}
\right.
\]
Consider the vertices of $\MP_{\MB_1}$ under $\Pi_1$. We have that the vertices 
\[
\Pi_1(2,3,k) = f_{2,1} + f_{3,k-2} \quad \text{for}\quad 4 \le k \le n
\]
are precisely the vertices which contain $f_{2,1}$. So the vertices of $\varphi_{(2,3)}\left( \Pi_1(\MP_{\MB_1}) \right)$ are the same as $\Pi_1(\MP_{\MB_1})$ except $\Pi(2,3,k) = f_{2,1} + f_{3,k-2}$ which are mapped to $f_{1,1} + f_{2,1} + f_{3,k-2}$. From the tables above we see that these are precisely the vertices of $\Pi^3_1\left( \MP_{\MB^3_1}\right)$. It follows that $\varphi_{(2,3)}$ induces a unimodular transformation taking $\MP_{\MB_1}$ to $\MP_{\MB^3_1}$.

\medskip

We now show that there is a combinatorial mutation taking $\MP_{\MB^{\lambda-1}_1}$ to $\MP_{\MB^\lambda_1}$. We define the tropical map $\varphi_{(2,\lambda)} = \varphi_{w_{(2,\lambda)}, F_{(2,\lambda)}}$ for each $4 \le \lambda \le n-1$ where
\[
w_{(2,\lambda)} = 
\Pi^3_1\left(
\begin{bmatrix}
0 & -1 & 0 & \dots & 0 &  1 & 0 & \dots & 0 \\
0 &  1 & 0 & \dots & 0 & -1 & 0 & \dots & 0 \\
0 &  0 & 0 & \dots & 0 & 0 & 0 & \dots & 0 
\end{bmatrix}
\right)
= f_{1,\lambda-2} + f_{2,1} - f_{2,\lambda-2}.
\]
In the above matrix the non-zero entries lie in columns $2$ and $\lambda$. We also have
\begin{align*}
F_{(2,\lambda)} &= 
\conv \left\{ \underline{0}, \ \Pi^3_1 \left(
\begin{bmatrix}
 0 &  0 &  1 & \dots &  1 &  0 & 0 & \dots & 0 \\
 0 & -1 & -1 & \dots & -1 & -1 & 0 & \dots & 0 \\
 0 &  0 &  0 & \dots & 0  &  0 & 0 & \dots & 0 \\
\end{bmatrix}
\right)
\right\}\\
&= \conv \left\{
\underline{0}, \ -f_{2,1} + \sum_{\mu = 1}^{\lambda-3} (f_{1, \mu} - f_{2,\mu+1})
\right\}.
\end{align*}
In the matrix in the definition of $F_{(2,\lambda)}$, the columns with entries $(0,-1,0)^T$ are $2$ and $\lambda$.

We now calculate the image of $\Pi^3_1\left( \MP_{\MB^{\lambda-1}_1}\right)$ under $\varphi_{(2,\lambda)}$. The vertices of the image are written down in Table~\ref{tab:vertices_under_phi_2_lambda}. We see that the only vertices which are changing by the tropical map are those corresponding to $(2,\lambda,k)$ for each $k \in \{\lambda + 1, \dots, n \}$. Furthermore, we see that the vertices of $\Pi^3_1\left(\MP_{\MB^{\lambda-1}_1}\right)$ are mapped by $\varphi_{(2,\lambda)}$ to the vertices of $\Pi^3_1\left(\MP_{\MB^\lambda_1}\right)$.

\begin{table}
    \centering
    \resizebox{\textwidth}{!}{
    \begin{tabular}{clll}\toprule
        $v \in V\left(\MP_{\MB^{\lambda - 1}_1}\right)$ 
        & $\Pi^3_1(v)$ 
        & $\varphi_{(2,\lambda)}\left(\Pi^3_1(v)\right)$
        & conditions \\
    \midrule
        $(2,1,k)$ & $f_{3,k-2}$ & $f_{3,k-2}$ & $3 \le k \le n-1$\\
        $(i,1,k)$ & $f_{1,i-2} + f_{3,k-2}$ & $f_{1,i-2} + f_{3,k-2}$ &  $3 \le i < k \le n$\\
        $(i,2,k)$ & $f_{1,i-2} + f_{2,1} + f_{3,k-2}$ & $f_{1,i-2} + f_{2,1} + f_{3,k-2}$ & $3 \le i < \lambda \leq k \le n$ \\
    \midrule
        $(2, \lambda ,k)$ & $f_{2,\lambda-2} + f_{3,k-2}$ & $f_{1,\lambda -2} + f_{2,1} + f_{3,k-2}$ & $\lambda < k \le n$\\
        $(2, j ,k)$ & $f_{2,j-2}+f_{3,k-2}$ & $f_{2,j-2}+f_{3,k-2}$ & $\lambda < j < k \le n$\\
        $(i, j ,k)$ & $f_{1,i-2}+f_{2,j-2}+f_{3,k-2}$ & $f_{1,i-2}+f_{2,j-2}+f_{3,k-2}$ & $3 \le i < j < k \le n$\\
    \bottomrule
    \end{tabular}
    }
    \caption{For each fixed $4 \le \lambda \le n-1$, we write down the vertices of the polytope $\varphi_{(2,\lambda)}\left(\Pi^3_1\left(\MP_{\MB^{\lambda-1}_1}\right)\right)$. Note that $(2,1,n)$ is mapped to $\underline{0}$ by both $\Pi^3_1$ and $\varphi_{(2,\lambda)}(\Pi^3_1)$.}
    \label{tab:vertices_under_phi_2_lambda}
\end{table}

And so, by Proposition~\ref{prop:convex}, we have constructed a combinatorial mutation between the polytopes corresponding to the matching fields $\MB^{\lambda -1}_1$ and $\MB^{\lambda}_1$ for each $\lambda \in \{4, \dots, n-1 \}$. 

\medskip

The final step is to show that the polytopes corresponding to $\MB^{n-1}_1$ and $\MB_1^n=\MB_2$ are unimodular equivalent. Consider the projection
\[
\Pi_2 = 
\begin{bmatrix}
 0 & 0 & f_{1,1} & f_{1,2} & \dots & f_{1,n-3} & 0 \\
 f_{2,1} & 0 & 0 & f_{2,2} & \dots & f_{2,n-3} & 0 \\
 0 & 0 & f_{3,1} & f_{3,2} & \dots & f_{3,n-3} & 0 \\
\end{bmatrix}.
\]
Let us write down the vertices of $\MP_{\MB_2}$ under the projection $\Pi_2$.

\begin{center}
    \begin{tabular}{llll}\toprule
        $v \in V(\MP_{\MB_2})$ & $\Pi_2(v)$ & conditions & $w \in V(\MP_{\MB^{n-1}_1})$\\
        \midrule
        $(1,2,k)$ & $f_{3,k-2}$ & $3 \le k \le n$ & $(2,1,k)$\\
        $(i,1,k)$ & $f_{1,i-2} +f_{2,1} + f_{3, k-2}$ & $3 \le i < k \leq n$ & $(i,2,k)$\\
        $(i,2,k)$ & $f_{1,i-2} + f_{3,k-2}$ & $3 \le i < k \le n$ & $(i,1,k)$ \\
        $(i,j,k)$ & $f_{1,i-1} + f_{2,j-2} + f_{3,k-2}$ & $3 \le i< j < k \le n$ & $(i,j,k)$ \\
        \bottomrule
    \end{tabular}
\end{center}
In the table we also record for each vertex $v$ of $\MP_{\MB_2}$ the corresponding vertex $w$ of $\MP_{\MB^{n-1}_1}$ such that $\Pi_2(v) = \Pi^3_1(w)$. These polytopes are therefore identical under their respective projections, hence they are unimodular equivalent.
\end{proof}

\begin{theorem}\label{thm:gr3n_mutate_ell_block}
Let $\ell \in \{3, \dots, n-1 \}$. Then there is a sequence of combinatorial mutations taking the polytope $\MP_{\MB_{\ell-1}}$ to $\MP_{\MB_{\ell}}$.
\end{theorem}

\begin{proof}
We begin by giving an overview of the structure of the proof, see Figure~\ref{fig:overview_mutate_ell}. We construct a sequence of combinatorial mutations taking the polytope $\MP_{\MB_0}$ to $\MP_{\MB_1}$ which passes through the intermediate polytopes $\MP_{\MB_0^\lambda} \subseteq \RR^{3 \times n}$ for each $2 \le \lambda \le n-1$. We do this by constructing projections $\Pi_0, \Pi_0^2$ and $\Pi_1$ from $\RR^{3 \times n}$ to $\RR^{3 \times (n-3)}$ and tropical maps $\varphi_{(1,\lambda)} : \RR^{3 \times (n-3)} \rightarrow \RR^{3 \times (n-3)}$. We show that the tropical maps lift to combinatorial mutations of the matching field polytopes.

\begin{figure}
    \centering
    \includegraphics[scale = 0.85]{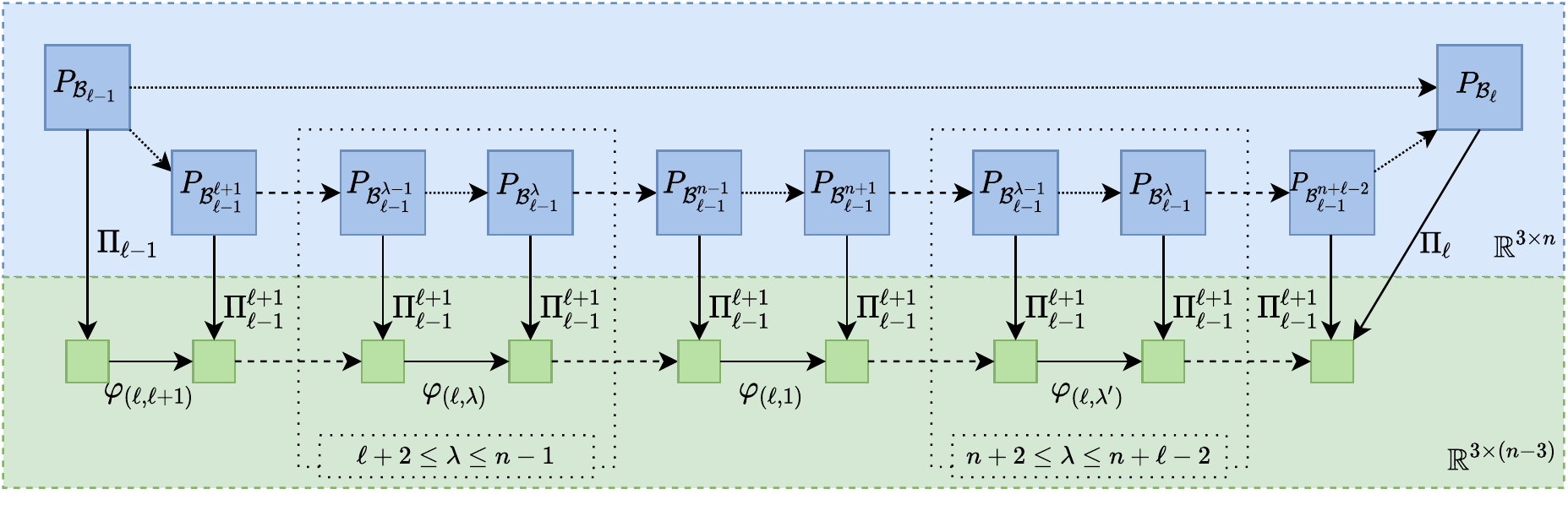} 
    \caption{Overview of the proof of Theorem~\ref{thm:gr3n_mutate_ell_block}. The shaded squares represent polytopes. The squares in the top region are matching field polytopes for block diagonal and intermediate matching fields. The squares in the lower region are the images of these polytopes under the given projections.}
    \label{fig:overview_mutate_ell}
\end{figure}

Let $V = \RR^{3 \times (n-3)}$ with canonical basis given by $\{f_{i,j} : 1 \le i \le 3, 1 \le j \le n-3 \}$ where $f_{i,j}$ is the matrix with $1$ in row $i$ and columns $j$ and zeros everywhere else. For each $\ell \in \{3, \dots, n-1 \}$ we define the projections
\[
\Pi_{\ell - 1} = 
\begin{bmatrix}
0 & f_{1,1} & f_{1,2} & \dots & f_{1,\ell-3} & 0 & f_{1,\ell - 2} & f_{1,\ell - 1} & f_{1,\ell} &\dots & f_{1,n-3} & 0 \\
 f_{2,1} & 0 & f_{2,2} & \dots & f_{2,\ell-3} & f_{2,\ell-2} & \mathbf{0} & \mathbf{f_{2,\ell - 1}} & f_{2,\ell} & \dots & f_{2,n-3} & 0 \\
 0 & 0 & f_{3,1} & \dots & f_{3,\ell-4} & f_{3,\ell-3} & f_{3,\ell - 2} & f_{3,\ell - 1} &f_{3,\ell} & \dots & f_{3,n-3} & 0 
\end{bmatrix},
\]
\[
\Pi^{\ell+1}_{\ell-1} = 
\begin{bmatrix}
0 & f_{1,1} & f_{1,2} & \dots & f_{1,\ell-3} & 0 & f_{1,\ell - 2} & f_{1,\ell - 1} & f_{1,\ell} & \dots & f_{1,n-3} & 0 \\
 f_{2,1} & 0 & f_{2,2} & \dots & f_{2,\ell-3} & f_{2,\ell-2} & \mathbf{f_{2,\ell - 1}} & \mathbf{0} & f_{2,\ell} & \dots & f_{2,n-3} & 0 \\
 0 & 0 & f_{3,1} & \dots & f_{3,\ell-4} & f_{3,\ell-3} & f_{3,\ell - 2} & f_{3,\ell - 1} & f_{3,\ell} & \dots & f_{3,n-3} & 0 
\end{bmatrix}.
\]
Let us write down the vertices of the projection of $\MP_{\MB_{\ell - 1}}$ under $\Pi_{\ell - 1}$.

\begin{center}
    \begin{tabular}{lll}\toprule
        $v \in V(\MP_{\MB_{\ell - 1}})$ & $\Pi_{\ell -1}(v)$ & conditions \\
        \midrule
        $(1,2,k)$ & $f_{3,k-2}$ & $3 \le k \le n$ \\
        $(1,j,k)$ & $f_{2,j-1} + f_{3,k-2}$ & $3 \le j < \ell \leq k \le n$ \\
        $(i,1,k)$ & $f_{1,i-2} + f_{2,1} + f_{3,k-2}$  & $\ell \le i < k \le n$ \\
    \midrule
        $(i,j,k)$ & $f_{1,i-1} + f_{2,j-1}+ f_{3,k-2}$ & $3 \le i < j < \ell \leq k \le n$\\
        $(j,i,k)$ & $f_{1, j-2} + f_{2,i-1} + f_{3,k-2}$ & $3 \le i < \ell \le j < k \le n-1$ \\
        $(i,j,k)$ & $f_{1,i-2} + f_{2,j-2} + f_{3,k-2}$ & $\ell \le i < j < k \le n-1$ \\
        \bottomrule
    \end{tabular}
\end{center}
Let us also write down the vertices of $\MP_{\MB^{\ell+1}_{\ell-1}}$ under $\Pi^{\ell+1}_{\ell-1}$.

\begin{center}
    \begin{tabular}{lll}\toprule
        $v \in V\left(\MP_{\MB^{\ell+1}_{\ell-1}}\right)$ & $\Pi^{\ell+1}_{\ell-1}(v)$ & conditions \\
    \midrule
        $(1,2,k)$ & $f_{3,k-2}$ & $3 \le k \le n$ \\
        $(1,j,k)$ & $f_{2,j-1} + f_{3,k-2}$ & $3 \le j < \ell \leq k \le n$ \\
        $(i,1,k)$ & $f_{1,i-2} + f_{2,1} + f_{3,k-2}$  & $\ell \le i < k \le n$ \\
    \midrule
        $(i,j,k)$ & $f_{1,i-1} + f_{2,j-1}+ f_{3,k-2}$ & $3 \le i < j < \ell \leq k \le n$\\
        $(j,i,k)$ & $f_{1, j-2} + f_{2,i-1} + f_{3,k-2}$ & $3 \le i < \ell \le j < k \le n$ \\
        $(\ell+1,\ell, k)$ & $f_{1,\ell - 1} + f_{2,\ell-1} + f_{3,k-2}$ & $\ell+1 < k \le n-1$\\
    \midrule
        $(\ell, j, k)$ & $f_{1,\ell-2} + f_{2,j-2} + f_{3,k-2}$ & $\ell + 1 < j < k \le n$ \\
        $(i,j,k)$ & $f_{1,i-2} + f_{2,j-2} + f_{3,k-2}$ & $\ell < i < j < k \le n$ \\
        \bottomrule
    \end{tabular}
\end{center}
We define the linear map $\varphi_{(\ell,\ell+1)} : V \rightarrow V$ as follows
\[
\varphi_{(\ell,\ell+1)} (f_{i,j}) = 
\left\{
\begin{tabular}{ll}
    $f_{1,\ell-1}- f_{1,\ell-2} + f_{2,\ell-1}$ & if $(i,j) = (2,\ell-1)$, \\
    $f_{i,j}$ & otherwise.
\end{tabular}
\right.
\]
It is clear that this map is unimodular, so let us consider its action on the vertices of the projection of the polytope corresponding to $\MB_{\ell-1}$. From the above table, we see that the only vertices which contain $f_{2,\ell-1}$ are those corresponding to $(\ell,\ell+1,k)$ in $\MB_{\ell-1}$. Note that the matching fields and the projections of the vertices of their polytopes, $\MB^{\ell+1}_{\ell-1}$ and $\MB_{\ell-1}$, coincide for all tuples except for $(\ell, \ell+1, k)$ in $\MB_{\ell-1}$ and $(\ell+1,\ell,k)$ in $\MB^{\ell+1}_{\ell-1}$. The map $\varphi_{(\ell,\ell+1)}$ takes the vertex $\Pi_{\ell-1}(\ell,\ell+1,k)$ to $\Pi^{\ell+1}_{\ell-1}(\ell+1,\ell,k)$.
Hence $\varphi_{(\ell,\ell+1)}$ induces a unimodular transformation from $\MP_{\MB_{\ell-1}}$ to $\MP_{\MB^{\ell+1}_{\ell-1}}$.

\medskip

We now construct tropical maps $\varphi_{(\ell,\lambda)} : V \rightarrow V$ which take the matching field polytope associated to $\MB^{\lambda-1}_{\ell-1}$ to $\MB^{\lambda}_{\ell-1}$ for each $\lambda \in \{\ell+2, \dots, n-1 \}$. We define the tropical map $\varphi_{(\ell,\lambda)} = \varphi_{w_{(\ell, \lambda)}, F_{(\ell,\lambda)}}$ as follows
\[
w_{(\ell,\lambda)} = \Pi^{\ell+1}_{\ell-1} \left(
\begin{bmatrix}
 0 & \dots & -1 & 0 & \dots & 0 &  1 & 0 & \dots 0 \\
 0 & \dots &  1 & 0 & \dots & 0 & -1 & 0 & \dots 0 \\
 0 & \dots &  0 & 0 & \dots & 0 &  0 & 0 & \dots 0
\end{bmatrix}
\right)
= -f_{1,\ell-2} + f_{1,\lambda-2} + f_{2,\ell-1} -f_{2,\lambda-2}
\]
where the non-zero entries of the above matrix lie in columns $\ell$ and $\lambda$. We also define
\begin{align*}
F_{(\ell,\lambda)} &= \conv \left\{
\underline{0}, \ \Pi^{\ell+1}_{\ell-1} \left(
\begin{bmatrix}
0 & \dots & 0 &  0 &  1 & \dots &  1 &  0 & 0 & \dots & 0  \\
0 & \dots & 0 & -1 & -1 & \dots & -1 & -1 & 0 & \dots & 0  \\
0 & \dots & 0 &  0 &  0 & \dots &  0 &  0 & 0 & \dots & 0  \\
\end{bmatrix}
\right)
\right\} \\
&= \conv \left\{
\underline{0}, \ -f_{2,\ell - 1}+ \sum_{\mu = \ell}^{\lambda-1} (f_{1,\mu} - f_{2,\mu})
\right\},
\end{align*}
where the columns of the above matrix with entries $(0,-1,0)^T$ are columns $\ell$ and $\lambda$.

We write down the image of the vertices of the polytope corresponding to $\MB^{\lambda-1}_{\ell-1}$ and their image under $\varphi_{(\ell,\lambda)}$ in Table~\ref{tab:vertices_under_phi_ell_lambda}. From the table we see that the only vertices which are changed by $\varphi_{(\ell, \lambda)}$ are those corresponding to $(\ell,\lambda,k)$ where $k \in \{\lambda+1, \dots, n \}$. The image of the vertex corresponding to $(\ell, \lambda, k)$ is the vertex corresponding to $(\lambda, \ell, k)$. These are precisely the vertices which differ between the polytopes of $\MB^{\lambda-1}_{\ell-1}$ and $\MB^\lambda_{\ell-1}$. And so, Proposition~\ref{prop:convex} implies that $\varphi_{(\ell, \lambda)}$ is a combinatorial mutation taking the polytope corresponding to $\MB^{\lambda-1}_{\ell-1}$ to $\MB^\lambda_{\ell-1}$.

\begin{table}
    \centering
    \resizebox{\textwidth}{!}{
    \begin{tabular}{clll}\toprule
        $v \in V\left(\MP_{\MB^{\lambda - 1}_{\ell-1}}\right)$ 
        & $\Pi^{\ell+1}_{\ell-1}(v)$ 
        & $\varphi_{(\ell,\lambda)}\left(\Pi^{\ell+1}_{\ell-1}(v)\right)$
        & conditions \\
    \midrule
        $(1,2,k)$ & $f_{3,k-2}$ & $f_{3,k-2}$ & $3 \le k \le n$ \\
        $(1,j,k)$ & $f_{2,j-1}+f_{3,k-2}$ & $f_{2,j-1}+f_{3,k-2}$ & $3 \le j < \ell \leq k \le n$ \\
        $(i,j,k)$ & $f_{1,i-1}+f_{2,j-1}+f_{3,k-2}$ & $f_{1,i-1}+f_{2,j-1}+f_{3,k-2}$ & $2 \le i < j < \ell \leq k \le n$ \\
    \midrule
        $(i,1,k)$ & $f_{1,i-2} +f_{2,1}+f_{3,k-2}$ & $f_{1,i-2} +f_{2,1}+f_{3,k-2}$ & $\ell \le i < k \le n$ \\
        $(i,2,k)$ & $f_{1,i-2} +f_{3,k-2}$ & $f_{1,i-2} +f_{3,k-2}$ & $\ell \le i < k \le n$ \\
        $(j,i,k)$ & $f_{1,j-2}+f_{2,i-1}+f_{3,k-2}$ & $f_{1,j-2}+f_{2,i-1}+f_{3,k-2}$ & $3 \le i<\ell \le j < k \le n$ \\
    \midrule
        $(i,\ell,k)$ & $f_{1,i-2} + f_{2,\ell-1}+f_{3,k-2}$ & $f_{1,i-2} + f_{2,\ell-1}+f_{3,k-2}$ & $\ell < i < \lambda \leq k \le n$ \\
        $(\ell,\lambda,k)$ & $f_{1,\ell-2}+f_{2,\lambda-2} + f_{3,k-2}$ & $f_{1,\lambda-2} + f_{2,\ell-1}+f_{3,k-2}$ & $\lambda < k \le n$ \\
        $(\ell,j,k)$ & $f_{1,\ell-2} + f_{2,j-2}+f_{3,k-2}$ & $f_{1,\ell-2} + f_{2,j-2}+f_{3,k-2}$ & $\ell < j < k \le n$ \\
    \midrule
        $(i,j,k)$ & $f_{1,i-2}+f_{2,j-2}+f_{3,k-2}$ & $f_{1,i-2}+f_{2,j-2}+f_{3,k-2}$ & $\ell < i < j < k \le n$ \\
    \bottomrule
    \end{tabular}
    }
    \caption{For each fixed $\lambda \in \{\ell+1, \dots, n-1\}$, we write down the vertices of the polytope $\varphi_{(\ell,\lambda)}\left(\Pi^{\ell+1}_{\ell-1}\left(\MP_{\MB^{\lambda-1}_{\ell-1}}\right)\right)$. }
    \label{tab:vertices_under_phi_ell_lambda}
\end{table}

\medskip

Now we go to the case $\lambda \in \{n+1,\ldots,n+\ell-1\}$. Let $\lambda'=\lambda-n$. 
We define the tropical maps $\varphi_{(\ell,\lambda')} = \varphi_{w_{(\ell,\lambda')}, F_{(\ell,\lambda')}}$
for $\lambda' \in \{1, \dots, \ell-2 \}$ as follows
\begin{align*}
w_{(\ell,\lambda')} &= \Pi^{\ell+1}_{\ell-1} \left(
\begin{bmatrix}
 0 & \dots &  1 & 0 & \dots & 0 & -1 & 0 & \dots 0 \\
 0 & \dots & -1 & 0 & \dots & 0 &  1 & 0 & \dots 0 \\
 0 & \dots &  0 & 0 & \dots & 0 &  0 & 0 & \dots 0
\end{bmatrix}
\right) \\
&= \left\{
\begin{tabular}{ll}
    $f_{1,1} -f_{1,\ell-2}  +f_{2,\ell-1} $ & if $\lambda'$ = 2 \\
    $f_{1,\lambda'-1} -f_{1,\ell-2} -f_{2,\lambda'-1} +f_{2,\ell-1} $ & otherwise.
\end{tabular}
\right.
\end{align*}
Note that the non-zero entries of the above matrix lie in columns $\lambda'$ and $\ell$. We also define
\begin{align*}
F_{(\ell,1)} &= \conv \left\{
\underline{0}, \ \Pi^{\ell+1}_{\ell-1} \left(
\begin{bmatrix}
 0 & 0 & \dots & 0 &  0 &  1 & \dots &  1  \\
-1 & 0 & \dots & 0 & -1 & -1 & \dots & -1  \\
 0 & 0 & \dots & 0 &  0 &  0 & \dots &  0  \\
\end{bmatrix}
\right)
\right\}, \\
F_{(\ell,\lambda')} &= \conv \left\{
\underline{0}, \ \Pi^{\ell+1}_{\ell-1} \left(
\begin{bmatrix}
0 & \dots & 0 & -1 & -1 & \dots & -1 & -1 & 0 & \dots & 0  \\
0 & \dots & 0 &  0 &  1 & \dots &  1 &  0 & 0 & \dots & 0  \\
0 & \dots & 0 &  0 &  0 & \dots &  0 &  0 & 0 & \dots & 0  \\
\end{bmatrix}
\right)
\right\} \ \textrm{for } \lambda' \ge 2,
\end{align*}
where the columns of $F_{(\ell,1)}$ with entries $(0,-1,0)^T$ are columns $1$ and $\lambda'$ and the columns of $F_{(\ell,\lambda')}$ with entries $(-1,0,0)^T$ are columns $\lambda'$ and $\ell$. 

The map $\varphi_{(\ell,1)}$ changes the vertices corresponding to $(1,\ell,k)$ where $k \in \{\ell+1, \dots, n \}$. As a result, the image of these vertices under $\varphi_{(\ell,1)}$ are vertices of $\MB^{n+1}_{\ell-1}$ corresponding to the tuples $(1,\ell,k)$. All other vertices are fixed by the tropical map.
Now, it follows by Proposition~\ref{prop:convex} that $\varphi_{(\ell,1)}$ is a combinatorial mutation from the polytope of $\MB^{n-1}_{\ell-1}$ to $\MB^{n+1}_{\ell-1}$.

Similarly, for the map $\varphi_{(\ell,\lambda')}$ where $\lambda' \in \{2, \dots, \ell-2 \}$, we can show that the only vertices of the polytope of $\MB^{\lambda-1}_{\ell-1}$ which change are those corresponding to the tuples $(\ell,\lambda',k)$ where $k \in \{\ell+1, \dots, n \}$. As a result, the image of these vertices are precisely the vertices of $\MB^{\lambda}_{\ell-1}$ corresponding to the tuples $(\lambda',\ell,k)$. Since these are precisely the vertices which differ between the matching fields  $\MB^{\lambda-1}_{\ell-1}$ and $\MB^\lambda_{\ell-1}$, it follows that $\varphi_{(\ell,\lambda')}$ induces a combinatorial mutation between these polytopes.

\medskip

The final step is to show that the matching field polytopes for $\MB^{n+\ell-2}_{\ell-1}$ under the projection $\Pi^{\ell+1}_{\ell-1}$ and the polytope for $\MB_\ell$ under the projection $\Pi_{\ell}$ are unimodular equivalent. We write down the vertices of these polytopes that do not coincide in the table below.

\begin{center}
\resizebox{\textwidth}{!}{
\begin{tabular}{ll|lll}
\toprule
    $v \in V(\MP_{\MB^{n+\ell-2}_{\ell-1}})$ & $\Pi^{\ell+1}_{\ell-1}(v)$ & $v' \in V(\MP_{\MB_\ell})$ & $\Pi_\ell(v')$ & conditions \\
\midrule
    $(\ell,\ell-1,k)$ & $f_{1,\ell-2} +f_{2,\ell-2}+f_{3,k-2}$ & $(\ell-1,\ell,k)$  & $f_{1,\ell-2} + f_{2,\ell-1}+f_{3,k-2}$ & $\ell+1 \le k \le n$ \\
\bottomrule
\end{tabular}
}
\end{center}
We define the linear map $\varphi_{(\ell,\ell-1)} : V \rightarrow V$ by
\[
\varphi_{(\ell,\ell-1)}(f_{i,j}) = 
\left\{
\begin{tabular}{ll}
    $f_{1,\ell-2} -f_{2,\ell-2}+f_{2,\ell-1}$ & if $(i,j) = (1,\ell-2)$, \\
    $f_{i,j}$ & otherwise. 
\end{tabular}
\right.
\]
Consider the action of this map, which is clearly unimodular, on the vertices of
$\MB^{n+\ell-2}_{\ell-1}$. Note that the only vertices that contain $f_{1,\ell-2}$ are those described in the table above, i.e.~those corresponding to $(\ell,\ell-1, k)$ for some $k$. Hence, the image of the vertices of $\MB^{n+\ell-2}_{\ell-1}$ under $\varphi_{(\ell,\ell-1)}$ coincide with the vertices of the polytope of $\MB_\ell$. Therefore, these polytopes are unimodular equivalent.
\end{proof}

\subsection{Convexity of mutations}\label{sec:mutations_convexity}

In this section we collect the results used throughout the above proofs. These results show that the image of the matching field polytopes under their respective tropical map $\varphi_{(\ell,\lambda)}$ have the same number of vertices as the original polytope. In particular, no new vertices are created as a result of applying the tropical map. Moreover, we show that all images are convex and so they are given by the convex hull of the images of the vertices. Figure~\ref{fig:mutation_convexity} shows that for an arbitrary polytope $P$ and tropical map $\varphi$, the image $\varphi(P)$ need not be convex and may have different number of vertices.

\begin{figure}
    \centering
    \resizebox{\textwidth}{!}{    \includegraphics{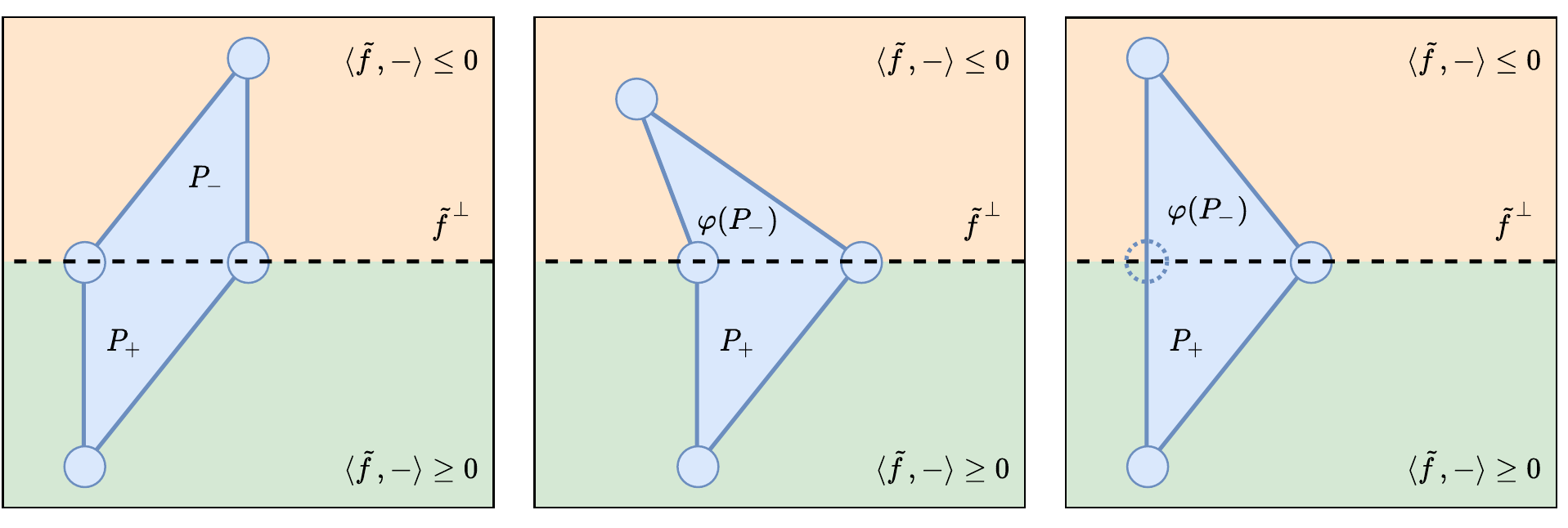}} 
    \caption{Depiction of what may happen when applying a tropical map $\varphi$ to a polytope. The left most diagram is the original polytope $P$ with the hyperplane $\tilde f^\perp$ diving the space into two regions, see Lemma~\ref{lem:inner_prod_F} for notation. The middle diagram shows how $\varphi(P)$ may fail to be convex. The rightmost diagram shows how $\varphi(P)$ may have different number of vertices to $P$.}
    \label{fig:mutation_convexity}
\end{figure}

\smallskip

\noindent\textbf{Notation.} Throughout this section we refer to the combinatorial mutations $\varphi_{(\ell, \lambda)}$ defined in Theorems~\ref{thm:gr3n_mutate_01_block}, \ref{thm:gr3n_mutate_12_block} or \ref{thm:gr3n_mutate_ell_block} where $\lambda \in \{\ell+2, \dots, n-1 \} \cup \{ 1, \dots, \ell - 1\}$. Note that in the theorems, the symbol $\lambda'$ is used for the values $\{1, \dots, \ell - 1 \}$ however in this section we will simply write $\lambda$.

\begin{lemma}\label{lem:inner_prod_F}
Let $\varphi_{(\ell,\lambda)} = \varphi_{w_{(\ell,\lambda)},F_{(\ell,\lambda)}}$ be a tropical map from Theorems~\ref{thm:gr3n_mutate_01_block}, \ref{thm:gr3n_mutate_12_block} or \ref{thm:gr3n_mutate_ell_block}. Suppose $F := F_{(\ell,\lambda)}= \conv\{\underline{0}, \tilde{f} \}$ and let $\Pi = \Pi^{\ell+1}_{\ell-1}$ be the projection from the same Theorem. For each vertex $u = (i,j,k)$ of the matching field polytope of $\MB^{\lambda-1}_{\ell-1}$ we have
\[
\langle \Pi(u), \Pi(\tilde{f}) \rangle =
\left\{
\begin{tabular}{ll}
    $-1$ & if $i = \ell, \ j = \lambda$,  \\
    $1$ & if $(i,j)\in \mathcal{A}_{(\ell,\lambda)}$,\\
    $0$ & otherwise.
\end{tabular}
\right.
\]
Here $\mathcal{A}_{(\ell,\lambda)}$ is the set consisting of $(i,j)$ for which $1\leq i,j\leq n$ and one of the following holds:
\begin{itemize}
    \item $\ell < i < \lambda < j$,
    \item $j < \ell < i < \lambda$,
    \item $i < \lambda < j < \ell$,
    \item $\lambda < j < \ell < i$.
\end{itemize}
\end{lemma}
\begin{proof}
We fix $V$ to be the vector space $\RR^{3 \times (n-3) }$ with canonical basis $f_{i,j}$ where $1 \le i \le 3$ and $1 \le j \le n-3$. To show this lemma holds, we exhaustively check that it holds for each $\ell$ and $\lambda$. In particular we partition the cases as follows
\begin{itemize}
    \item $\ell = 1$ and $3 \le \lambda \le n-2$,
    \item $2 \le \ell < \lambda \le n-1$ and $\lambda \ge \ell+2$,
    \item $3 \le \ell \le n-2$ and $\lambda = 1$,
    \item $4 \le \ell \le n-2$ and $2 \le \lambda \le i-2$.
\end{itemize}

\smallskip

\noindent\textbf{Case 1.} Assume that $\ell = 1$ and $3 \le \lambda \le n-2$. We have
\begin{align*}
\Pi^2_0 &= 
\begin{bmatrix}
 f_{1,1} & 0 & f_{1,2} & f_{1,3} & \dots & f_{1, n-3} & 0 & 0 \\
 0 & 0 & f_{2,1} & f_{2,2} & \dots & f_{2, n-4} & f_{2, n-3} & 0 \\
 0 & 0 & f_{3,1} & f_{3,2} & \dots & f_{2, n-4} & f_{3, n-3} & 0 
\end{bmatrix}, \\
F_{(1,\lambda)} &= \conv
\left\{
\underline{0}, \ 
\Pi^2_0 \left(
\begin{bmatrix}
-1 & 0 & \dots & 0 & -1 & -1 & \dots & -1 \\
 0 & 0 & \dots & 0 &  0 &  1 & \dots & 1  \\
 0 & 0 & \dots & 0 &  0 &  0 & \dots & 0  
\end{bmatrix}
\right)
\right\}\\
&= \conv \left\{
\underline{0}, \ -f_{2,\ell - 1} +\sum_{\mu = \ell}^{\lambda-1} (f_{1,\mu} - f_{2,\mu})
\right\} \\
&= \conv\{\underline{0}, \tilde{f} \}.
\end{align*}
In the above matrix corresponding to $\tilde{f}$, the columns with entries $(-1,0,0)$ are columns one and $\lambda$.  Consider the vertices of the matching field polytope for $\MB^{\lambda-1}_0$. We observe the following
\begin{itemize}
    \item $\langle u, \tilde{f} \rangle = -1$ if and only if $u = (1,\lambda,k)$ where $k \ge \lambda+1$,
    \item $\langle u, \tilde{f} \rangle = 1$ if and only if $u = (i,j,k)$ where $1 < i < \lambda < j < k$.
\end{itemize}

\noindent\textbf{Case 2.} Assume that $2 \le \ell < \lambda \le n-1$ and $\lambda \ge \ell+2$. We have
\begin{align*}
\Pi^3_1 &= 
\begin{bmatrix}
0 & 0 & f_{1,1} & f_{1,2} & \cdots & f_{1,n-3} & 0 \\
0 & f_{2,1} & 0 & f_{2,2} & \cdots & f_{2,n-3} & 0 \\
0 & 0 & f_{3,1} & f_{3,2} & \cdots & f_{3,n-3} & 0 
\end{bmatrix} ,\\
\Pi^{\ell+1}_{\ell-1} &= 
\begin{bmatrix}
0 & f_{1,1} & f_{1,2} & \dots & 0 & f_{1,\ell - 2} & f_{1,\ell - 1} & \dots & f_{1,n-3} & 0 \\
 f_{2,1} & 0 & f_{2,2} & \dots & f_{2,\ell-2} & f_{2,\ell - 1} & 0 & \dots & f_{2,n-3} & 0 \\
 0 & 0 & f_{3,1} & \dots & f_{3,\ell-3} & f_{3,\ell - 2} & f_{3,\ell - 1} & \dots & f_{3,n-3} & 0 
\end{bmatrix} \textrm{ for } \ell \ge 3, \\
F_{(\ell,\lambda)} &= \conv \left\{
\underline{0}, \ \Pi^{\ell+1}_{\ell-1} \left(
\begin{bmatrix}
0 & \dots & 0 &  0 &  1 & \dots &  1 &  0 & 0 & \dots & 0  \\
0 & \dots & 0 & -1 & -1 & \dots & -1 & -1 & 0 & \dots & 0  \\
0 & \dots & 0 &  0 &  0 & \dots &  0 &  0 & 0 & \dots & 0  \\
\end{bmatrix}
\right)
\right\} \\
&= \conv \left\{
\underline{0}, \ -f_{2,\ell - 1} +\sum_{\mu = \ell}^{\lambda-1} (f_{1,\mu} - f_{2,\mu})
\right\} \\
&= \conv\{\underline{0}, \tilde{f} \}.
\end{align*}
In the above matrix corresponding to $\tilde{f}$, the columns with entries $(0,-1,0)$ are columns $\ell$ and $\lambda$. Consider the vertices of the matching field polytope for $\MB^{\lambda-1}_{\ell-1}$. We observe the following
\begin{itemize}
    \item $\langle u, \tilde{f} \rangle = -1$ if and only if $u = (\ell,\lambda,k)$ where $k \ge \lambda+1$,
    \item $\langle u, \tilde{f} \rangle = 1$ if and only if $u = (i,j,k)$ where $j < \ell < i < \lambda$ or $\ell < i < \lambda < j$.
\end{itemize}

\noindent\textbf{Case 3.} Assume that $3 \le \ell \le n-2$ and $\lambda = 1$. We have $\Pi^{\ell+1}_{\ell-1}$ as defined above.
\begin{align*}
F_{(\ell,1)} &= \conv \left\{
\underline{0}, \ \Pi^{\ell+1}_{\ell-1} \left(
\begin{bmatrix}
 0 & 0 & \dots & 0 &  0 &  1 & \dots &  1  \\
-1 & 0 & \dots & 0 & -1 & -1 & \dots & -1  \\
 0 & 0 & \dots & 0 &  0 &  0 & \dots &  0  \\
\end{bmatrix}
\right)
\right\} \\
&= \conv\{\underline{0}, \tilde{f} \}.
\end{align*}
In the above matrix corresponding to $\tilde{f}$, the columns with entries $(0,-1,0)$ are columns one and $\ell$. Consider the vertices of the matching field polytope for $\MB^{n-1}_{\ell-1}$. We observe the following
\begin{itemize}
    \item $\langle u, \tilde{f} \rangle = -1$ if and only if $u = (\ell,1,k)$ where $k \ge \ell+1$,
    \item $\langle u, \tilde{f} \rangle = 1$ if and only if $u = (i,j,k)$ where $1 < j < \ell < i$.
\end{itemize}

\noindent\textbf{Case 4.} Assume that $4 \le \ell \le n-2$ and $2 \le \lambda = i-2$. We have $\Pi^{\ell+1}_{\ell-1}$ as defined above.
\begin{align*}
F_{(\ell,\lambda)} &= \conv \left\{
\underline{0}, \ \Pi^{\ell+1}_{\ell-1} \left(
\begin{bmatrix}
0 & \dots & 0 & -1 & -1 & \dots & -1 & -1 & 0 & \dots & 0  \\
0 & \dots & 0 &  0 &  1 & \dots &  1 &  0 & 0 & \dots & 0  \\
0 & \dots & 0 &  0 &  0 & \dots &  0 &  0 & 0 & \dots & 0  \\
\end{bmatrix}
\right)
\right\} \\
&= \conv\{\underline{0}, \tilde{f} \}.
\end{align*}
In the above matrix corresponding to $\tilde{f}$, the columns with entries $(-1,0,0)$ are columns $\lambda$ and $\ell$. Consider the vertices $u$ of the matching field polytope for $\MB^{n-1}_{\ell-1}$. We observe the following
\begin{itemize}
    \item $\langle u, \tilde{f} \rangle = -1$ if and only if $u = (\ell,1,k)$ where $k \ge \ell+1$,
    \item $\langle u, \tilde{f} \rangle = 1$ if and only if $u = (i,j,k)$ where $i< \lambda < j < \ell $ or $\lambda < j < \ell < i$.
\end{itemize}
\end{proof}

We now show that, for each matching field polytope in the proofs of the above theorems, there are no edges which pass through the $\tilde{f}^\perp$.

\begin{lemma}\label{lem:no_edge_in_matching_poly}
Take $\ell, \lambda, \tilde{f}$ and $\Pi$ as in the statement of Lemma~\ref{lem:inner_prod_F}. Suppose $u, v$ are vertices of the matching field polytope of $\MB^{\lambda-1}_{\ell-1}$ such that $\langle\Pi(u), \Pi(\tilde{f}) \rangle = -1$ and $\langle\Pi(v), \Pi(\tilde{f}) \rangle = 1$. Then the segment from $u$ to $v$ is not an edge of the matching field polytope of $\MB^{\lambda-1}_{\ell-1}$.
\end{lemma}

\begin{proof}
In general, if $\MP \subset \RR^d$ is a polytope and $a,b,c,d \in V(P)$ are vertices such that $a+b = c+d$, then it follows that the segment from $a$ to $b$ is not an edge of $P$.

\smallskip

By Lemma~\ref{lem:inner_prod_F} we have that $u = (\ell,\lambda,k)$ for some $k \ge \max\{\ell,\lambda \} + 1$. Let us write $v = (i,j,k')$. 
By Lemma~\ref{lem:inner_prod_F} we have that $(i,j) \in \mathcal A_{(\ell, \lambda)}$, and so by definition we have one of the following cases.

\medskip

\noindent\textbf{Case 1.} Assume that $\ell < i < \lambda < j$. Then we have
\[
(\ell,\lambda,k) + (i,j,k') = (i,\lambda,k) + (\ell,j,k').
\]

\noindent\textbf{Case 2.} Assume that $j < \ell < i < \lambda$.
Then we have
\[
(\ell,\lambda,k) + (i,j,k') = (i,\lambda,k) + (\ell,j,k').
\]

\noindent\textbf{Case 3.} Assume that $i < \lambda < j < \ell$.
Then we have
\[
(\ell,\lambda,k) + (i,j,k') = (\ell,j,k) + (i,\lambda,k').
\]

\noindent\textbf{Case 4.} Assume that $\lambda < j < \ell < i$.
Then we have
\[
(\ell,\lambda,k) + (i,j,k') = (\ell,j,k) + (i,\lambda,k').
\]

Note that in each case, the summands appearing on the right hand side of the equation are also vertices of $\MB^{\lambda-1}_{\ell}$.
\end{proof}

We now show that no new vertices are created by applying the map $\varphi_{\ell,\lambda}$.

\begin{lemma}\label{lem:no_extra_vertices_in_image}
Let $\MP \subset \RR^d$ be a rational polytope, $\tilde{f} \in \RR^d$ be a non-zero vector. Suppose that $\MP \subset \{x \in \RR^d : -1 \le \langle x, \tilde{f} \rangle \le 1 \}$. Write $\MP_+ = \{ x \in \MP : 0 \le \langle x, \tilde{f} \rangle \le 1 \}$ and $\MP_- = \{x \in \MP: -1 \le \langle x, \tilde{f} \rangle \le 0 \}$. Let $w$ be a non-zero vector such that $\tilde{f} \in w^\perp$ and define $F = \conv(\underline{0}, \tilde{f})$. Let $\varphi = \varphi_{w,F}$ be the tropical map and note that $\varphi(\MP) = \varphi(\MP_+) \cup \varphi(\MP_-)$ is the union of two polytopes that intersect along a common face. 
Let $V(\varphi(\MP))$ be the vertices of $\varphi(\MP)$.
Then, there exists a vertex $v \in V(\varphi(\MP))$ that is not the image of any vertex of $\MP$ if and only if there exist vertices $u, v \in V(\MP)$ such that $\langle u, \tilde{f} \rangle = -1 $, $\langle v, \tilde{f} \rangle = 1 $ and the segment from $u$ to $v$ is an edge of $\MP$.  
\end{lemma}

\begin{proof}
On the one hand, suppose that $x \in V(\varphi(\MP))$ is a vertex not contained in $\varphi(V(\MP))$. Since $\varphi$ is an affine map on $\tilde{f}_+ := \{v : \langle v,\tilde{f} \rangle \ge 0 \}$ and $\tilde{f}_- := \{v : \langle v,\tilde{f} \rangle \le 0 \}$, it follows that $x \in \tilde{f}^\perp$. In particular, $x \in \tilde{f}_+$ hence $\varphi(x) = x$ and so $x \in \MP$. By assumption $x$ is not a vertex of $\MP$ so let $G$ be the face of $\MP$ containing $x$ in its relative interior.
Since $x$ lies in the relative interior of $G$, it follows that $G$ is not properly contained in $\tilde{f}_+$ or $\tilde{f}_-$. We see that $\varphi(G)$ is the union of two faces, namely 
\[
\varphi(G) = \varphi(G \cap \tilde{f}_+) \cup \varphi(G \cap \tilde{f}_-) \subseteq \varphi(\MP_+) \cup \varphi(\MP_-). 
\]
It follows that $\varphi(G \cap \tilde{f}^\perp)$ is a face of $\varphi(\MP_+)$ and $\varphi(\MP_-)$. However, $x$ is contained in the interior of $\varphi(G \cap \tilde{f}^{\perp})$ and by assumption $x$ is a vertex of $\varphi(\MP)$. Hence $G \cap \tilde{f}^\perp =\{x\}$. Since $\tilde{f}^\perp$ is a hyperplane not containing $G$, we have that $\dim(G) = \dim(G \cap \tilde{f}^\perp) + 1$. Therefore $\dim(G) = 1$ and so $G$ is an edge which passes from the interior of $\tilde{f}_+$ to the interior of $\tilde{f}_-$. Since $\MP \subset \{v : -1 \le \langle v, \tilde{f} \rangle \le 1 \}$ and the vertices of $\MP$ are rational, it follows that $V(G) = \{u,v \}$ where $\langle u, \tilde{f} \rangle = -1$ and $\langle v, \tilde{f} \rangle = 1$.

\medskip

On the other hand, suppose $u$ and $v$ are vertices of $\MP$ such that $\langle u, \tilde{f} \rangle = -1$ and $\langle v, \tilde{f} \rangle = 1$. If the segment from $u$ to $v$, call it $e = [u,v]$, is an edge of $\MP$ then we have that $e \cap \tilde{f}_+$ and $\varphi(e) \cap \tilde{f}_-$ are edges in the $1$-skeleton of $\varphi(\MP)$ with common vertex $\{x\} = e \cap \tilde{f}^\perp$. Since $x$ is fixed by $\varphi$ and lies in the relative interior of $e$, it is not a vertex of $\MP$. Hence $x \in V(\varphi(\MP))$ and $x \not\in V(\MP)$.
\end{proof}

\begin{remark}
Note that in the above proof $\varphi(\MP)$ is not assumed to be a polytope. Instead it is treated as the union of two polytopes that share a common face. In the case that the common face is a vertex, the original polytope is a line segment and the result trivially holds. So, we may assume that the common face has dimension at least $2$. In general the notion of the $1$-skeleton of these objects is well-defined, even though they may not be polytopes.
\end{remark}

\begin{lemma}\label{lem:no_edge_in_image}
Let $\ell, \lambda, \tilde{f}$ and $\Pi$ be as in the statement of Lemma~\ref{lem:inner_prod_F}. Let $\varphi = \varphi_{\ell,\lambda}$ be the tropical map. Suppose that $u$ and $v$ are vertices of the matching field polytope $\MP$ corresponding to $\MB^{\lambda-1}_{\ell-1}$ such that $\langle u, f \rangle = -1$ and $\langle v, f \rangle = 1$. Then the line segment from $\varphi(u)$ to $\varphi(v)$ lies in $\varphi(P)$.
\end{lemma}

\begin{proof}
We begin by noting that $\varphi$ acts on the vertices of $\MP$ by sending $(\ell,\lambda, k)$ to $(\lambda, \ell, k)$ for each $k \ge \ell+1$ and fixing all other vertices. By Lemma~\ref{lem:inner_prod_F}, $u = (\ell,\lambda,k)$ for some $k$ and $v = (i,j,k')$ where $v$ belongs to the set $\mathcal{A}_{(\ell,\lambda)}$. 
The proof proceeds similarly to the proof of Lemma~\ref{lem:no_edge_in_matching_poly}. We will show that there exist vertices $u'$ and $v'$ of $\varphi(\MP)$ such that $\varphi(u)+\varphi(v) = u' + v'$ and $\langle u', \tilde{f} \rangle = \langle v', \tilde{f}\rangle = 0 $. It follows from this that the segment from $\varphi(u)$ to $\varphi(v)$ is not an edge of $\varphi(\MP)$.

\smallskip

We now proceed by taking cases on the elements of $\mathcal{A}_{(\ell,\lambda)}$
in Lemma~\ref{lem:inner_prod_F}.

\medskip

\noindent\textbf{Case 1.} Assume that $\ell < i < \lambda < j$. Then we have
\[
(\lambda,\ell,k) + (i,j,k') = (i,\ell,k) + (\lambda,j,k').
\]
We show that $(i,\ell,k)$ and $(\lambda, j, k')$ are indeed vertices of $\varphi(\MP)$ as follows. In the matching field $\MB^{\lambda}_{\ell-1}$, corresponding to $\varphi(\MP)$, we have that $i < \lambda$, hence $(i,\ell,k)$ is a vertex. We also have that $\ell < \lambda < j < k'$ hence $(\lambda,j,k')$ is a vertex.

\noindent\textbf{Case 2.} Assume that $j < \ell < i < \lambda$.
Then we have
\[
(\lambda,\ell,k) + (i,j,k') = (\lambda,j,k) + (i,\ell,k').
\]
We show that $(\lambda, j, k)$ and $(i,\ell,k')$ are indeed vertices of $\varphi(\MP)$ as follows. Note that $\ell < \lambda$, so in the matching field $\MB^{\lambda}_{\ell-1}$, corresponding to $\varphi(\MP)$, we have that $i < \lambda$, hence $(i,\ell,k')$ is a vertex. We also have that $j < \ell < \lambda < k$  hence $(\lambda,j,k)$ is a vertex.

\noindent\textbf{Case 3.} Assume that $i < \lambda < j < \ell$.
Then we have
\[
(\lambda,\ell,k) + (i,j,k') = (i,\ell,k) + (\lambda,j,k').
\]
We show that $(i,\ell,k)$ and $(\lambda,j,k')$ are indeed vertices of $\varphi(\MP)$ as follows. Note that $\lambda < \ell$, so for the matching field $\MB^{\lambda}_{\ell-1}$, corresponding to $\varphi(\MP)$, we have that $i < \lambda$, hence $(i,\ell,k)$ is a vertex. Since $(i,j,k')$ is a vertex and we have that $\lambda < j < \ell < k'$ hence $(\lambda,j,k')$ is a vertex.

\noindent\textbf{Case 4.} Assume that $\lambda < j < \ell < i$.
Then we have
\[
(\lambda,\ell,k) + (i,j,k') = (\lambda,\ell,k) + (i,\ell,k').
\]
We show that $(\lambda,\ell,k)$ and $(i,\ell,k')$ are indeed vertices of $\varphi(\MP)$ as follows. Note that $\lambda < \ell$, so for the matching field $\MB^{\lambda}_{\ell-1}$, corresponding to $\varphi(\MP)$, we have that $j < \ell$, hence $(\lambda,j,k')$ is a vertex. Since $(i,j,k')$ is a vertex and we have that $j < \ell < i$ hence $(i,\ell,k')$ is a vertex.
\end{proof}

\begin{proposition}\label{prop:convex}
Let $\MP$ be the polytope of the matching field $\MB^\lambda_{\ell-1}$. Let $\varphi = \varphi_{(\ell,\lambda)}$ be the tropical map. Then $\varphi(\MP)$ is convex. Moreover its vertices are $\varphi(V(\MP))$, the image of the vertices of $P$.
\end{proposition}

\begin{proof}
Let $F =F_{\ell,\lambda}= \conv\{\underline{0}, \tilde{f}\}$ be the factor of the combinatorial mutation. We let $\MP_+ = \{x \in \MP : \langle x, \tilde{f} \rangle \ge 0  \}$ and $\MP_- = \{x \in \MP : \langle x, \tilde{f} \rangle \le 0 \}$. By Lemma~\ref{lem:inner_prod_F} we have that all edges of $\MP$ lie either in $\MP_+$ or $\MP_-$ and so it follows that $\MP \cap \tilde{f}^\perp = \conv\{ v \in V(\MP) : v \in \tilde{f}^\perp\}$, $\MP_+ = \conv\{v \in V(\MP) : \langle v, \tilde{f} \rangle \ge 0 \}$ and $\MP_- = \conv\{v \in V(\MP) : \langle v, \tilde{f} \rangle \le 0 \}$.

Note that the tropical map fixes all points in $\MP_+$. To see that $\varphi(\MP)$ is convex, by Lemma~\ref{lem:no_edge_in_image}, we have that the line segment from $\varphi(u)$ to $\varphi(v)$ lies in $\varphi(\MP)$ for each pair of vertices $u,v \in \MP$ such that $\langle u, \tilde{f} \rangle = -1$ and $\langle v, \tilde{f} \rangle = 1$.

To see that the vertices of $\varphi(\MP)$ are precisely the image of the vertices of $P$, note that by Lemma~\ref{lem:no_extra_vertices_in_image}, there are no additional vertices in $\varphi(\MP)$.
\end{proof}

\begin{example}[Continuation of Example~\ref{ex:diag} and \ref{exa:b_ell_1}]
For Grassmannian $\Gr(3,5)$, we consider the sequence of combinatorial mutations taking $\MP_{\MB_0}$, the Gelfand-Tsetlin polytope, to the matching field polytope $\MP_{\MB_1}$ given in the proof of Theorem~\ref{thm:gr3n_mutate_01_block}. The proof transforms the polytope as follows 
\[
\MP_{\MB_0} \xrightarrow{\sim}
\MP_{\MB_0^2} \rightarrow
\MP_{\MB_0^3} \rightarrow
\MP_{\MB_1}.
\]
Since $\MP_{\MB_0}$ is unimodular equivalent to $\MP_{\MB_0^2}$, we consider the map $\varphi_{(1,3)}:P_{\MB_0^2} \rightarrow \MP_{\MB_0^3}$. In Figure~\ref{fig:mutation_01_02} we illustrate the action of $\varphi_{(1,3)}$ on the vertex-edge graph of $\MP_{\MB_0^2}$ and compare it to the graph of $\MP_{\MB_0^3}$. The polytopes were calculated in Polymake \cite{polymake:2017}.


\begin{figure}
    \centering
    \resizebox{\textwidth}{!}{    \includegraphics{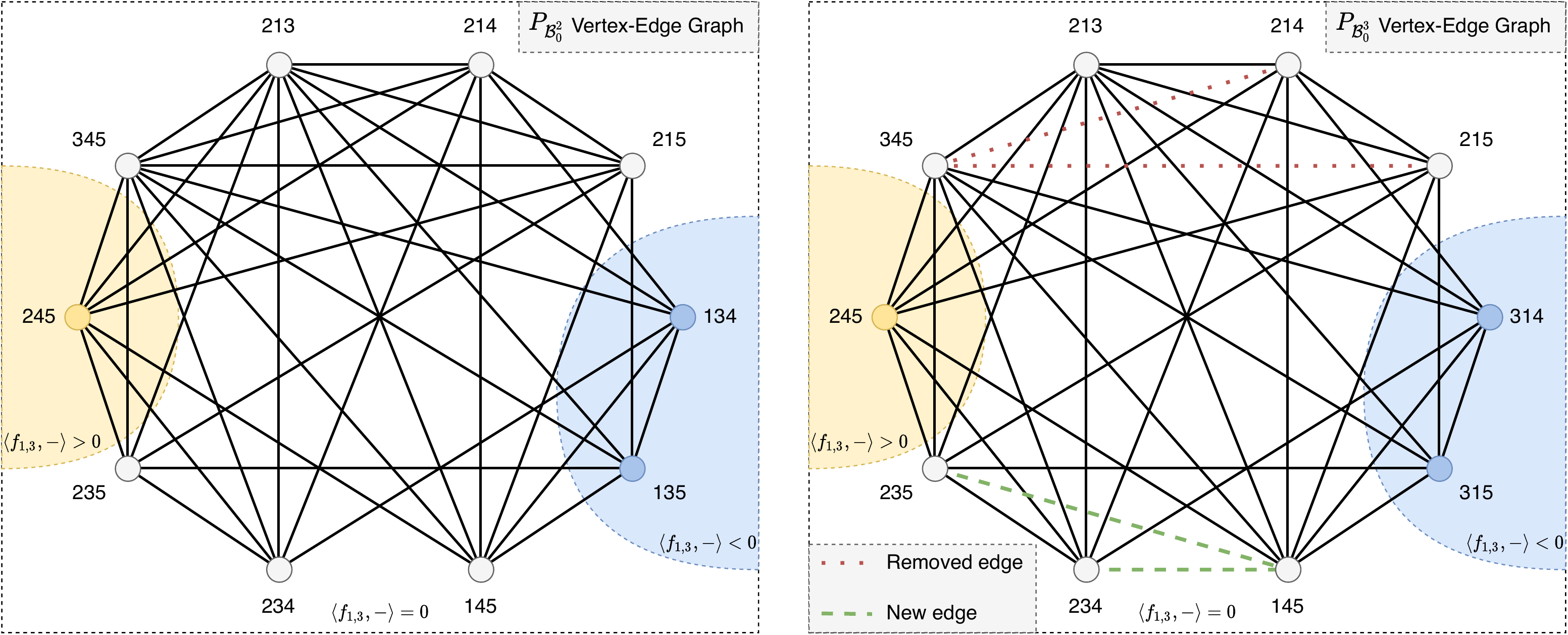}}
    \caption{The vertex-edge graphs of the polytopes $\MP_{\MB_0^2}$ and $\MP_{\MB_0^3}$. The combintorial mutation $\varphi_{(1,3)}$ transforms $\MP_{\MB_0^2}$ into $\MP_{\MB_0^3}$ by sending $134, 135 \in V(\MP_{\MB_0^2})$ to $314, 315$ respectively and fixing all other vertices. The dotted and dashed lines show which edges removed and created by the mutation. The shaded regions indicate the open half spaces on each side of the hyperplane $f_{1,3}^\perp = \{x : \langle f_{1,3}, x \rangle = 0\}$.}
    \label{fig:mutation_01_02}
\end{figure}
\end{example}

\subsection{Generalisation to \texorpdfstring{$\Gr(k,n)$}{Gr(k,n)}}\label{sec:mutation_generalise_gkn}

\begin{corollary}
Fix $k$ and $n$. The block diagonal matching field polytopes for $\Gr(k,n)$ are combinatorial mutation equivalent.
\end{corollary}

\begin{proof}
If $k = 3$ then the result holds by Theorems~\ref{thm:gr3n_mutate_01_block}, \ref{thm:gr3n_mutate_12_block} and \ref{thm:gr3n_mutate_ell_block}.
We now generalise these results for higher $k$ by extending the projection maps, matching fields and tropical maps used in the proofs of these theorems. 

\medskip

For each projection $\Pi : \RR^{3 \times (n-k+3)} \rightarrow \RR^{3 \times (n-k)}$ for the $\Gr(3,n-k+3)$ case we construct its analogous projection $\Pi' : \RR^{k \times n} \rightarrow \RR^{k \times (n-k)}$ for the $\Gr(k,n)$ case. We view $\RR^{3 \times (n-k)} \subset \RR^{k \times (n-k)}$ as a subspace and so we let $f_{i,j}$ denote the basis of $\RR^{k \times (n-k)}$ which extends the basis for $\RR^{3 \times (n-k)}$ defined in the proofs for the $\Gr(3, n-k+3)$ case.

To construct $\Pi'$, we join to the right hand side of $\Pi$ a zero matrix of size $3 \times (k - 3)$.
Next, for each $i \in \{4, \dots, k \}$, we join an extra row $R_i$ to the bottom of $\Pi$ where
\[
(R_i)_j = \left\{
\begin{tabular}{ll}
    $f_{i, j-i+1}$ & if $ j-i+1 \in \{1, 2, \dots, n-k \},$ \\
    $0$ & otherwise.
\end{tabular}
\right.
\]

\medskip

Note that there is a bijection between the block diagonal matching fields for $\Gr(k,n)$ and $\Gr(3,n-k+3)$ given by sending $\MB_{\ell}$ for $\Gr(k,n)$ to $\MB_{\ell}$ for $\Gr(3,n-k+3)$. Note that in $\Gr(k,n)$, if $\ell \ge n-k+2$ then the matching field $\MB_{\ell}$ is simply the diagonal matching field.

We define the intermediate matching fields for $\Gr(k,n)$ as follows. Let $\MB$ be any intermediate matching field for $\Gr(3,n-k+3)$. Then the corresponding matching field for $\Gr(k,n)$ has tuples
\[
\{ (i_1, i_2, \dots, i_k) :
(i_1, i_2, i_3) \in \MB, \ i_3 < i_4 < \dots < i_k
\}.
\]
Then, for any tropical map $\varphi$ taking the polytope of an intermediate matching field $\MB$ to the polytope of $\MB'$ for $\Gr(3,n-k+3)$, we define its analogue $\varphi'$ for $\Gr(k,n)$. This map acts by
\[
\varphi'(f_{i,j}) = 
\left\{
\begin{tabular}{ll}
    $\varphi(f_{i,j})$ & if $i \le 3$, \\
    $f_{i,j}$ & if $i > 3$. 
\end{tabular}
\right.
\]
Recall that $\RR^{3 \times (n-k)} \subseteq \RR^{k \times (n-k)}$ is a subspace. So if $\varphi = \varphi_{w,F}$ for some vector $w$ and polytope $F$, then the same $w$ and $F$ define $\varphi'$. Hence $\varphi'$ is also a tropical map. 
\end{proof}

\begin{example}[Extension to $k = 5$]
We show how to extend the results for $\Gr(3,6)$ to $\Gr(5,8)$. Let us begin by examining the matching fields $\MB_1$ and $\MB_2$. The tuples of the matching fields are shown in the table below. Note that for each row in the table, the tuples share the same first two entries.

\begin{center}
    \begin{tabular}{llllll}
    \toprule
        & $\Gr(3,6)$ & & & $\Gr(5,8)$ & \\
        & tuple & conditions & & tuple & conditions \\
    \midrule
        \multirow{2}{*}{$\MB_1$} & 
        $(j,1,k)$ & $2 \le j < k \le 6$ & &
        $(j,1,k_1,k_2,k_3)$ & $2 \le j < k_1 < k_2 < k_3 \le 8$ \\
        & $(i,j,k)$ & $2 \le i < j < k \le 6$ & &
        $(i,j,k_1,k_2,k_3)$ & $2 \le i < j < k_1 < k_2 < k_3 \le 8$ \\
    \midrule
        \multirow{3}{*}{$\MB_2$} &
        $(1,2,k)$ & $3 \le k \le 6$ & &
        $(1,2,k_1,k_2,k_3)$ & $3 \le k_1 < k_2 < k_3 \le 8$ \\
        & $(j,2,k)$ & $3 \le j < k \le 6$ & &
        $(j,2,k_1,k_2,k_3)$ & $3 \le j < k_1 < k_2 < k_3 \le 8$ \\
        & $(i,j,k)$ & $3 \le i < j < k \le 6$ & &
        $(i,j,k_1,k_2,k_3)$ & $3 \le i < j < k_1 < k_2 < k_3 \le 8$ \\
    \bottomrule
    \end{tabular}
\end{center}

For $\Gr(3,6)$, let us consider the sequence to combinatorial mutations taking the polytope $\MP_{\MB_1}$ to the block diagonal matching field polytope $\MP_{\MB_2}$.
The projection maps $\Pi$ in the proof of Theorem~\ref{thm:gr3n_mutate_12_block} are of the form
\[
\Pi = 
\begin{bmatrix}
 * & * & * & * & * & * \\
 * & * & * & * & * & * \\
 0 & 0 & f_{3,1} & f_{3,2} & f_{3,3} & 0 \\
\end{bmatrix}.
\]
For each such projection map we define the analogous projection map $\Pi'$ for $\Gr(5,8)$ as follows.
\[
\Pi' = 
\begin{bmatrix}
 * & * & * & * & * & * & 0 & 0\\
 * & * & * & * & * & * & 0 & 0 \\
 0 & 0 & f_{3,1} & f_{3,2} & f_{3,3} & 0 & 0 & 0 \\
 0 & 0 & 0 & f_{4,1} & f_{4,2} & f_{4,3} & 0 & 0 \\
 0 & 0 & 0 & 0 & f_{5,1} & f_{5,2} & f_{5,3} & 0 \\
\end{bmatrix}.
\]
We define the combinatorial mutations taking $\MP_{\MB_1}$ to $\MP_{\MB_2}$ for $\Gr(5,8)$ similarly. Take a tropical maps $\varphi_{w,F}$ in the proof of Theorem~\ref{thm:gr3n_mutate_12_block}. For some projection $\Pi$, we write
\[
w = \Pi \left(
\begin{bmatrix}
 * & * & * & * & * & * \\
 * & * & * & * & * & * \\
 0 & 0 & 0 & 0 & 0 & 0 \\
\end{bmatrix}
\right), \quad
F = \conv\left\{
\underline 0,  \ \Pi\left(
\begin{bmatrix}
 * & * & * & * & * & * \\
 * & * & * & * & * & * \\
 0 & 0 & 0 & 0 & 0 & 0 \\
\end{bmatrix}
\right)
\right\}.
\]
The corresponding tropical map for $\Gr(5,8)$ is given by $\varphi_{w',F'}$ where
\[
w' = \Pi' \left(
\begin{bmatrix}
 * & * & * & * & * & * & 0 & 0 \\
 * & * & * & * & * & * & 0 & 0 \\
 0 & 0 & 0 & 0 & 0 & 0 & 0 & 0 \\
 0 & 0 & 0 & 0 & 0 & 0 & 0 & 0 \\
 0 & 0 & 0 & 0 & 0 & 0 & 0 & 0 \\
\end{bmatrix}
\right), \quad
F' = \conv\left\{
\underline 0,  \ \Pi'\left(
\begin{bmatrix}
 * & * & * & * & * & * & 0 & 0 \\
 * & * & * & * & * & * & 0 & 0 \\
 0 & 0 & 0 & 0 & 0 & 0 & 0 & 0 \\
 0 & 0 & 0 & 0 & 0 & 0 & 0 & 0 \\
 0 & 0 & 0 & 0 & 0 & 0 & 0 & 0 \\
\end{bmatrix}
\right)
\right\}.
\]
The only non-zero entries of the above matrices, that define the tropical map, are in the first two rows. Therefore, the proof that $\MP_{\MB_1}$ is taken to $\MP_{\MB_2}$ by a sequence of combinatorial mutation in the $\Gr(3,6)$ case immediately applies to $\Gr(5,8)$. Note that the intermediate matching fields $\MB_\ell^\lambda$, similarly to $\MB_1$ and $\MB_2$, have the property that $(i_1,i_2,i_3)$ is a tuple in $\MB_\ell^\lambda$ for $\Gr(3,6)$ if and only if for any $i_3 < i_4 < i_5 \le 8$ we have that $(i_1, i_2, i_3, i_4, i_5)$ is a tuple in $\MB_\ell^\lambda$ for $\Gr(5,8)$.
\end{example}

\bibliographystyle{abbrv}
\bibliography{References.bib}

\bigskip
\noindent
\footnotesize 
{\bf Authors' addresses:}

\medskip

{\small\noindent School of Mathematics, University of Bristol, Bristol, BS8 1TW, UK\\
E-mail address: {\tt oliver.clarke@bristol.ac.uk }

\medskip\noindent
Department of Pure and Applied Mathematics, 
Osaka University, Suita, Osaka 565-0871, Japan\\
E-mail address:  {\tt higashitani@ist.osaka-u.ac.jp}

\medskip\noindent Department of Mathematics: Algebra and Geometry, Ghent University, 9000 Gent, Belgium \\
Department of Mathematics and Statistics, 
The Arctic University of Norway, 9037 Troms\o, Norway
\\ E-mail address: {\tt fatemeh.mohammadi@ugent.be}
}
\end{document}